\documentclass[a4paper,12pt,reqno]{amsart}
\usepackage{upref,amsmath,amsfonts,amsxtra,bbm,amssymb,mathrsfs,geometry}
\usepackage{enumitem}
\usepackage{hyperref}
\usepackage[english]{babel}

\usepackage{tikz}
\usetikzlibrary{calc}
\usetikzlibrary{arrows.meta,calc,decorations.pathreplacing}

\definecolor{myred}{RGB}{255, 100, 100}
\definecolor{myblue}{RGB}{100, 150, 255}
\definecolor{mygreen}{RGB}{40,180,95}
\definecolor{myorange}{RGB}{255,170,80}

\definecolor{skyblue}{HTML}{55e1ff}
\definecolor{limegreen}{HTML}{00FF00}
\definecolor{deeppink}{HTML}{FF00A2}
\definecolor{teal}{HTML}{008080}
\definecolor{light brown}{HTML}{D2B48C}

\usepackage{caption}
\usepackage{subcaption}
\usepackage{float}
\usepackage{xcolor}
\usepackage{ifthen}
\usepackage{siunitx}
\usetikzlibrary{patterns}
\usepackage[dvipsnames]{xcolor}

\hypersetup{
	colorlinks=true,
	linkcolor=blue,
	citecolor=blue,
	urlcolor=blue
}

\newcommand{\R}{{\mathbb{R}}}
\newcommand{\T}{{\mathbb{T}}}
\newcommand{\Z}{{\mathbb{Z}}}
\newcommand{\N}{{\mathbb{N}}}
\newcommand{\Q}{{\mathbb{Q}}}

\newcommand{\Om}{{\Omega}}
\newcommand{\La}{{\Lambda}}
\newcommand{\la}{{\lambda}}

\newcommand{\vep}{{\varepsilon}}
\newcommand{\m}{{\mathfrak{m}}}
\newcommand{\lrangle}[1]{\left\langle #1 \right\rangle}
\newcommand{\Zcomb}[2]{#1\mathbb{Z} + #2\mathbb{Z}}

\newcommand{\supp}{\operatorname{supp}}
\newcommand{\bd}{\text{bd}}

\newcommand{\qae}{\quad\text{a.e.}\ }

\newtheorem{theorem}{Theorem}[section]
\newtheorem{lemma}[theorem]{Lemma}
\newtheorem{proposition}[theorem]{Proposition}
\newtheorem{corollary}[theorem]{Corollary}
\newtheorem{remark}[theorem]{Remark}

\numberwithin{equation}{section}

\begin{document}
\baselineskip 18pt

\title{Weak tiling by unions of non-overlapping unit cubes}

\author{Tianyu Chen}
\address
{Tianyu Chen: School of Mathematics and Statistics, and Hubei Key Lab--Math. Sci., Central China Normal University, Wuhan 430079, China}
\email{chentianyu1024@gmail.com}

\author
{Shilei Fan}
\address
{Shilei FAN: School of Mathematics and Statistics, and Key Lab NAA--MOE, Central China Normal University, Wuhan 430079, China}
\email{slfan@ccnu.edu.cn}

\author{Mihail N. Kolountzakis}
\address{Mihail N. Kolountzakis: Department of Mathematics and Applied Mathematics, University of Crete, Voutes Campus, 70013 Heraklion, Greece, and Institute of Computer Science, Foundation for Research and Technology Hellas, N. Plastira 100, Vassilika Vouton, 700 13, Heraklion, Greece}
\email{kolount@gmail.com}

\author{Chun-Kit  Lai}
\address{Chun-Kit Lai: Department of Mathematics, San Francisco State University, San Francisco, CA 94132, USA.}
{}
\email{cklai@sfsu.edu}


\begin{abstract}
We study weak tiling by finite unions of pairwise non-overlapping unit cubes in \(\mathbb R^d\) that are not necessarily axis-parallel. We obtain a weak-tiling analogue of Keller’s classical theorem, which gives a structural restriction on the support of any weak tiling measure associated with the unit cube. This allows us to derive geometric restrictions on configurations of cubes whose union admits a weak tiling. As an application, we prove Fuglede's conjecture for unions of three non-overlapping unit squares in \(\mathbb R^2\) as well as unions of two non-overlapping axis-parallel unit cubes in any dimension.
\end{abstract}

\thanks{This work is supported by the NSF of Xinjiang Uygur Autonomous Region (Grant No. 2024D01A160) and the NSFC (grants No. 12331004  and No. 12231013).}
\maketitle

\setcounter{tocdepth}{1}

\tableofcontents  

\section{Introduction}

Let $\Om\subset\R^d$ be a bounded, measurable set with finite positive Lebesgue measure. We say that $\Om$ is a {\bf spectral set} if there exists a countable set $\La\subset\R^d$ such that the system of exponential functions $\{e^{2\pi i\lrangle{\la,x}}\}_{\la\in\La}$ forms an orthogonal basis in $L^2(\Om).$  In this case, the pair $(\Omega, \Lambda)$ is called a \textbf{spectral pair}, and $\Lambda$ is referred to as a \textbf{spectrum} for $\Omega$.
We say that $\Om$ is a \textbf{translational tile} of $\R^d$ if there exists a countable set $T\subset \R^d$ such that $$\sum_{t\in T}\mathbbm{1}_\Om(x-t)=1,\qae x \in \mathbb{R}^d.$$
T is called  a \textbf{tiling set} of $\Om$.

Fuglede~\cite{Fuglede1974} conjectured that $\Om$ is a spectral set if and only if $\Om$ is a tile. Fuglede proved the equivalence under the additional assumption that the spectrum or the tiling set of $\Om$ is a lattice. Later on, several partial results were established, for example, for 2-dimensional convex bodies~\cite{IKT03}, unions of two intervals~\cite{Laba2001}, certain unions of three intervals~\cite{BKA2010,BM2014} and three-dimensional convex polytopes~\cite{GL17}. The general conjecture has been shown to fail in both directions when \( d \ge 3 \) (see~
\cite{FMM07,FBR06,KM06,KMM06,M05,T04}). Recently, it was also disproved in $\R^2$ \cite{zhang2026},  leaving it open only for $\R^1$.

On the other hand, Fuglede's conjecture was also confirmed under extra assumptions. One prominent result was due to Lev and Matolcsi \cite{LevMat2022} who proved the conjecture for convex bodies in all dimensions. A central tool in their proof is the concept of \textbf{weak tiling}. Let $\Sigma\subset\R^d$ be a measurable set. We say that $\Sigma$ admits a \textbf{weak tiling} by translates of $\Om$ if there exists a positive, locally finite measure $\mu$ on $\R^d$ such that $$\mathbbm{1}_\Om*\mu=\mathbbm{1}_\Sigma\qae$$
And we say that $\Om$ is a \textbf{weak tile} of $\R^d$ if $\Om^c$ admits a weak tiling by translates of $\Om$. In other words, there exists a positive, locally finite measure $\mu$ such that $$\mathbbm{1}_\Om*\mu=\mathbbm{1}_{\Om^c}\qae$$
Using \cite[Corollary~2.6]{LevMat2022}, one obtains the following equivalent formulation: a bounded measurable set $\Omega\subset\R^d$ is a weak tile of $\R^d$ if and only if there exists a {non-negative}, locally finite measure $\nu$ satisfying
	\[
	\mathbbm{1}_\Omega * \nu = 1 \qae
	\quad\text{and}\quad
	\nu(\{0\})=1.
	\]
The following theorem establishes a link between spectrality and weak tiling. 
\begin{theorem}[\cite{LevMat2022}]\label{thm:1.1}
	Suppose $\Om$ is a bounded, measurable set in $\R^d$. If $\Om$ is spectral, then it is a weak tile. 
\end{theorem}
Clearly, every tile is a weak tile. Together with Theorem \ref{thm:1.1}, we know that both spectral sets and tiles belong to the class of weak tiles. Due to the counterexamples of Fuglede's conjecture in both directions, the class of weak tiles forms a strictly larger class than each of the classes of spectral sets and tiles.  Kiss, Londner, Matolcsi, and Somlai~\cite{KLMS2026} constructed an example of a ``lonely" weak tile which is neither a spectral set nor a tile, showing that the class of weak tiles is genuinely a broader class than the other two.

This now opens up new questions to decide  if all these three classes of sets are the same under some more assumptions on the sets. Lev and Matolcsi pioneered this study by showing that every spectral set weakly tiles its complement and using this necessary condition to establish Fuglede's conjecture for convex bodies \cite{LevMat2022}. Together with the subsequent result of Kolountzakis, Lev, and Matolcsi~\cite{KLM2023} that every convex weak tile is a tile, this shows that spectrality, weak tiling, and translational tiling are equivalent for convex bodies. Another such class is given by unions of at most two intervals in $\mathbb{R}$: weak tiling implies proper translational tiling for these sets, and, together with {\L}aba's result~\cite{Laba2001}, this yields the equivalence of the three notions \cite{KLM2025}.

\subsection{Main Results.} In this paper, we aim to explore the equivalence of these concepts for some finite unions of cubes in $\R^d$. We find two general results that may be of independent interest in further study. They include a generalization of Keller's classical criterion and the non-spectrality of a union of non-parallel convex polytopes.  

To begin, let us set up some notation for this paper. The open unit cube $(0,1)^d$ in $\R^d$ will be denoted by $Q_d$. We will write  $\Z^{\ast} = \Z\setminus\{0\}$, and 
\begin{equation}\label{gridL}
{\mathcal L} = \{x = (x_1,\cdots, x_d)\in\R^d: \exists i \  \mbox{such that} \ x_i\in\Z^{\ast} \}.
\end{equation}
In 1930, Keller~\cite{keller1930} showed that if $Q_d$ tiles $\R^d$ by translations with a tiling set ${\mathcal J}$, then for all $t\ne t'\in{\mathcal J}$, at least one of the coordinates of $t-t'$ must belong to $\Z^*$. Assuming $0\in {\mathcal J}$, Keller's theorem can be rephrased as ${\mathcal J}\subset {\mathcal L}\cup\{0\}$. Our first result is that Keller's theorem is also true for the weak tiling measures of $Q_d$.
\begin{theorem} [Keller's theorem for weak tiling]\label{th-Keller}
    Let $\mu$ be a positive locally finite measure such that 
    $$
    {\mathbbm 1}_{Q_d}\ast (\mu+\delta_0) = 1 \qae
    $$
    Then $\supp ({\mu}) \subset {\mathcal L}$. 
\end{theorem}

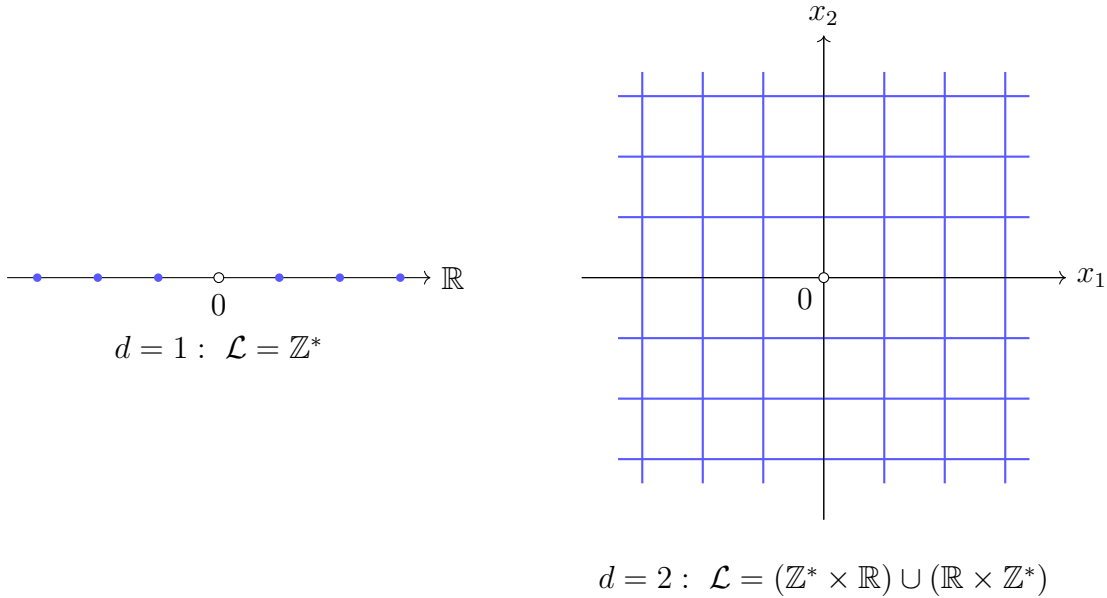
\begin{figure}[!ht]
\centering
\begin{tikzpicture}[scale=0.8]
  \begin{scope}
    \draw[->] (-3.5,0) -- (3.5,0) node[right] {$\mathbb R$};
    \foreach \x in {-3,-2,-1,1,2,3}
      \fill[blue!65] (\x,0) circle (2pt);
    \fill[white] (0,0) circle (2.3pt);
    \draw (0,0) circle (2.3pt);
    \node at (0,-0.45) {$0$};

    \node at (0,-1.15) {$d=1:\ \mathcal L=\mathbb Z^*$};
  \end{scope}

  \begin{scope}[xshift=10cm]
    \draw[->] (-4,0) -- (4,0) node[right] {$x_1$};
    \draw[->] (0,-4) -- (0,4) node[above] {$x_2$};

    \foreach \x in {-3,-2,-1,1,2,3}
      \draw[thick,blue!65] (\x,-3.4) -- (\x,3.4);
    \foreach \y in {-3,-2,-1,1,2,3}
      \draw[thick,blue!65] (-3.4,\y) -- (3.4,\y);

    \draw[thin] (-4,0) -- (4,0);
    \draw[thin] (0,-4) -- (0,4);

    \fill[white] (0,0) circle (2.3pt);
    \draw (0,0) circle (2.3pt);
    \node[below left] at (0,0) {$0$};

    \node at (0,-5) {$d=2:\ \mathcal L=(\mathbb Z^*\times\mathbb R)\cup(\mathbb R\times\mathbb Z^*)$};
  \end{scope}
\end{tikzpicture}
\caption{The support restriction in Keller's theorem for weak tiling measures. When \(d=1\), the support of \(\mu\) is contained in \(\mathbb Z^\ast\); in fact, \(\mu=\delta_{\mathbb Z^\ast}\). When \(d=2\), it is contained in $\mathcal L=(\mathbb Z^\ast\times\mathbb R)\cup(\mathbb R\times\mathbb Z^\ast)$, the union of all horizontal and vertical lines on which one coordinate is a nonzero integer. In particular, the origin does not belong to
\(\mathcal L\).}
\label{fig:keller-support}
\end{figure}

{We say that a family of measurable sets is \textbf{pairwise non-overlapping} if every two distinct sets intersect only on a set of measure zero.} Now we consider the pairwise non-overlapping union
	\[
	\Omega = Q\cup P_1\cup\cdots\cup P_N\subset \R^d,
	\]
where we write \(Q:=Q_d\) for brevity and the \(P_i\) are {full-dimensional} convex polytopes (i.e. they have positive volume). A {\bf facet} is a \((d-1)\)-dimensional face of a polytope. When we say that \(F\) is a facet of \(\Omega\), we mean that \(F\) is a facet of one of the polytopes in the union. Each facet \(F\) lies in a \((d-1)\)-dimensional hyperplane with a unit outward normal vector \(n_F\). The standard unit vectors in \(\mathbb R^d\) are denoted by \(e_1,\ldots,e_d\).

Our second general result shows that if none of the polytopes \(P_i\) has a facet whose outward normal vector is parallel to any coordinate axis, then \(\Omega\) cannot be a spectral set. This extends our understanding of the Fuglede conjecture beyond the axis-parallel setting. See Figure~\ref{fig:non-spectral-polytopes}.

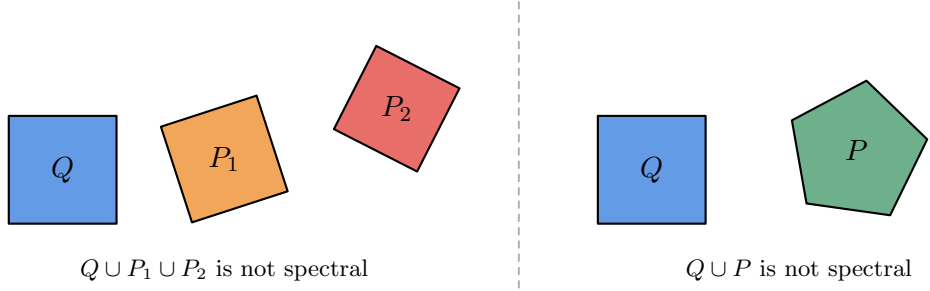
\begin{figure}[!htbp]
	\centering
	\definecolor{cubeblue}{RGB}{100,155,230}
	\definecolor{squareorange}{RGB}{242,166,90}
	\definecolor{squarered}{RGB}{232,110,105}
	\definecolor{polygongreen}{RGB}{110,175,140}

	\begin{tikzpicture}[scale=0.95,line join=round]

		\begin{scope}
			\filldraw[fill=cubeblue,draw=black,line width=0.8pt]
			(0,0) rectangle (1.5,1.5);
			\node[font=\small] at (0.75,0.75) {$Q$};

			\begin{scope}[shift={(3.0,0.9)},rotate=18]
				\filldraw[fill=squareorange,draw=black,line width=0.8pt]
				(-0.7,-0.7) rectangle (0.7,0.7);
				\node[font=\small] at (0,0) {$P_1$};
			\end{scope}

			\begin{scope}[shift={(5.4,1.6)},rotate=-27]
				\filldraw[fill=squarered,draw=black,line width=0.8pt]
				(-0.65,-0.65) rectangle (0.65,0.65);
				\node[font=\small] at (0,0) {$P_2$};
			\end{scope}

			\node[font=\scriptsize,align=center] at (3.0,-0.65)
			{\(Q\cup P_1\cup P_2\) is not spectral};
		\end{scope}

		\draw[densely dashed,gray!65,line width=0.7pt] (7.1,-0.9)--(7.1,3.15);

		\begin{scope}[shift={(8.2,0)}]
			\filldraw[fill=cubeblue,draw=black,line width=0.8pt]
			(0,0) rectangle (1.5,1.5);
			\node[font=\small] at (0.75,0.75) {$Q$};

			\filldraw[fill=polygongreen,draw=black,line width=0.8pt]
			(4.585,1.174) --
			(3.739,1.990) --
			(2.701,1.438) --
			(2.905,0.281) --
			(4.069,0.117) -- cycle;

			\node[font=\small] at (3.60,1.02) {$P$};

			\node[font=\scriptsize,align=center] at (2.8,-0.65)
			{\(Q\cup P\) is not spectral};
		\end{scope}

	\end{tikzpicture}

	\caption{
		Two non-spectral sets covered by
		Theorem~\ref{th-spectral-non-parallel}.
		In both examples, none of the facets of the additional
		polytopes is parallel to a coordinate axis.
	}
	\label{fig:non-spectral-polytopes}
\end{figure}


\begin{theorem}\label{th-spectral-non-parallel}
    Let $\Omega=Q~\cup P_1~\cup\cdots\cup P_N$ as above. Suppose that for all facets in $P_i$, $i = 1,\cdots, N$, the unit normal vectors are not equal to $\pm e_j$, $j=1,\cdots, d$. Then $\Omega$ is not a spectral set. 
\end{theorem}

\noindent {\bf (i) Axis-parallel cubes.} We now use our general theorems to discuss our findings concerning unions of several unit cubes. Using the Keller criterion (Theorem \ref{th-Keller}), we can establish the validity of the Fuglede spectral set conjecture for the union of two non-overlapping axis-aligned cubes.

\begin{theorem}\label{th-two-cubes}
Let $A = \{0,a\}\subset\R^d$ and $\Omega = A+Q_d$ with $Q_d\cap (Q_d+a) = \varnothing$. Then the following are equivalent.
    \begin{enumerate}[label=(\arabic*),font=\normalfont]
        \item $\Omega$ is a weak tile.
        \item $\Omega$ is a tile.
        \item $\Omega$ is a spectral set.
        \item $a\in {\mathcal L}$.
    \end{enumerate}
\end{theorem}

We can move on to discuss results with three unit intervals in $\R^1$ or squares in $\R^2$.

\begin{theorem}\label{th-three-cubes-parallel}
Suppose that $\Omega$ is a union of three disjoint unit intervals in $\R^1$ or is a union of three disjoint axis-parallel unit squares in $\R^2$. If $\Omega$ is a weak tile in $\R^d$, then it is both a tile and a spectral set in $\R^d$ for $d = 1,2$. 
\end{theorem}

Combining this with Theorem \ref{thm:1.1}, we conclude that weak tiles, tiles and spectral sets are in the same class for a union of three disjoint unit intervals in $\R^1$ or a union of three disjoint axis-parallel unit squares in $\R^2$.

In $\R$, the spectral set conjecture for three intervals of the same length was solved in \cite{BKA2010}, and it was recently solved by Shi~\cite{Shi2026} for three intervals of arbitrary lengths. Neither of these works addresses weak tiling for unions of three intervals. Theorem \ref{th-three-cubes-parallel} gives a new proof for the result in \cite{BKA2010} by completely classifying weak tiles that are unions of three intervals of equal length.

We now briefly talk about the proof strategy in $\R^2$. A structural consequence of Theorem~\ref{th-Keller}, stated in Corollary~\ref{cor-equal-size-necessary-dimension-2}, shows that, after possibly interchanging the coordinate axes, the three squares lie in horizontal strips at integer heights. We may therefore apply the following reduction principle for layered sets, which preserves the three properties relevant to our argument: translational tiling, weak tiling, and spectrality.

\begin{theorem}\label{thm:layered-reduction}
Let \(S\subset\mathbb Z\) be finite, and let
\[
\Omega
=
\bigcup_{j\in S}\Omega_j\times(j,j+1)
\subset\mathbb R^2,
\qquad
\widetilde{\Omega}
=
\bigcup_{j\in S}\Omega_j\times\{j\}
\subset\mathbb R\times\mathbb Z,
\]
where each \(\Omega_j\subset\mathbb R\) is bounded and measurable with finite positive Lebesgue measure.
Then \(\Omega\) is a tile/weak tile/spectral set in \(\mathbb R^2\) if and only if
    \(\widetilde{\Omega}\) is a tile/weak tile/spectral set in
    \(\mathbb R\times\mathbb Z\).
\end{theorem}

Theorem~\ref{thm:layered-reduction} reduces the parallel-square problem to the study of  unions of three pairwise non-overlapping unit intervals in \(\mathbb R\times\mathbb Z\). In Theorem \ref{prop-WT-3intervals-characterization}, we completely characterize the weak tiles in $\R\times\Z$  that are unions of three pairwise non-overlapping unit intervals and prove that every such set is both a translational tile and a spectral set. Applying Theorem~\ref{thm:layered-reduction} once more transfers these conclusions back to \(\mathbb R^2\), thereby completing the parallel case of Theorem~\ref{thm:Fuglede-conjecture-3squares}.

\begin{remark}
    {\rm{The conclusions of Theorems \ref{th-two-cubes} and \ref{th-three-cubes-parallel} do not extend to an arbitrary number of cubes due to the counterexamples of Fuglede's conjecture, which are finite unions of unit cubes. However, \cite[Theorem 2.4]{KM06} shows that if $\Om:=A+[0,1)^d$ is spectral, where $A\subset\Z^d$ and $|A|\le 5$, then $\Om$ tiles $\R^d$.  It will be interesting to determine the minimal number of cubes for which the conjecture starts to fail.} 
    }

\end{remark}

{
\noindent {\bf (ii) Non-axis-parallel cubes.} Because of Theorem \ref{th-spectral-non-parallel}, we can also say something about non-overlapping unit squares that are not necessarily parallel to the coordinate axes.  Indeed, by Theorem~\ref{th-spectral-non-parallel}, we know that if a union of two non-overlapping unit squares in $\R^2$ is a spectral set, then the two squares must be parallel to each other. Therefore, by Theorem~\ref{th-two-cubes} and Proposition~\ref{lem:congruent-square-tiling-parallel}, we can conclude that Fuglede's conjecture holds for the union of two non-overlapping unit squares in $\R^2$.
\begin{corollary}\label{cor-two-squares}
	Fuglede's conjecture holds for the union of two non-overlapping unit squares in $\R^2$. 
\end{corollary}}

{\rm 
\begin{remark} We have not resolved a similar version of the corollary in higher dimensions because Theorem~\ref{th-spectral-non-parallel} does not cover the case where one of the additional polytopes has even a single facet parallel to a coordinate hyperplane. See Figure~\ref{fig:two-cubes-one-rotated}.
\end{remark}
}

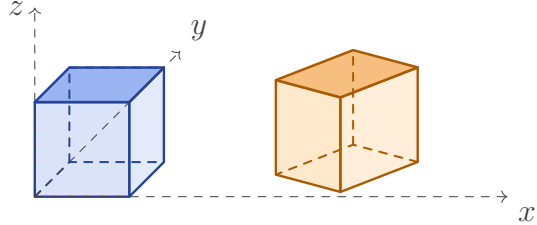
\begin{figure}[htbp]
\centering
\begin{tikzpicture}[
    x={(1cm,0cm)},
    y={(0.366cm,0.366cm)},
    z={(0cm,1cm)},
    line join=round,
    line cap=round,
    scale=1.25
]
    \definecolor{cubedraw}{RGB}{35,70,155}
    \definecolor{cubefill}{RGB}{190,205,245}
    \definecolor{facetblue}{RGB}{120,155,235}
    \definecolor{polydraw}{RGB}{170,90,0}
    \definecolor{polyfill}{RGB}{255,215,165}
    \definecolor{facetorange}{RGB}{245,170,85}

    \coordinate (A000) at (0,0,0);
    \coordinate (A100) at (1,0,0);
    \coordinate (A110) at (1,1,0);
    \coordinate (A010) at (0,1,0);
    \coordinate (A001) at (0,0,1);
    \coordinate (A101) at (1,0,1);
    \coordinate (A111) at (1,1,1);
    \coordinate (A011) at (0,1,1);

    \filldraw[
        fill=cubefill,
        fill opacity=.55,
        draw=cubedraw,
        line width=.9pt
    ]
        (A000)--(A100)--(A101)--(A001)--cycle;

    \filldraw[
        fill=cubefill,
        fill opacity=.42,
        draw=cubedraw,
        line width=.9pt
    ]
        (A100)--(A110)--(A111)--(A101)--cycle;

    \filldraw[
        fill=facetblue,
        fill opacity=.75,
        draw=cubedraw,
        line width=.95pt
    ]
        (A001)--(A101)--(A111)--(A011)--cycle;

    \draw[cubedraw,dashed,line width=.75pt]
        (A000)--(A010)--(A110);

    \draw[cubedraw,dashed,line width=.75pt]
        (A010)--(A011);

    \draw[cubedraw,dashed,line width=.75pt]
        (A011)--(A111);

    \pgfmathsetmacro{\cx}{3}
    \pgfmathsetmacro{\cy}{0.82}
    \pgfmathsetmacro{\h}{0.5}

    \pgfmathsetmacro{\ang}{60}
    \pgfmathsetmacro{\cosa}{cos(\ang)}
    \pgfmathsetmacro{\sina}{sin(\ang)}

    \coordinate (B1) at
        ({\cx-\h*\cosa+\h*\sina},
         {\cy-\h*\sina-\h*\cosa},0);

    \coordinate (B2) at
        ({\cx+\h*\cosa+\h*\sina},
         {\cy+\h*\sina-\h*\cosa},0);

    \coordinate (B3) at
        ({\cx+\h*\cosa-\h*\sina},
         {\cy+\h*\sina+\h*\cosa},0);

    \coordinate (B4) at
        ({\cx-\h*\cosa-\h*\sina},
         {\cy-\h*\sina+\h*\cosa},0);

    \coordinate (C1) at
        ({\cx-\h*\cosa+\h*\sina},
         {\cy-\h*\sina-\h*\cosa},1);

    \coordinate (C2) at
        ({\cx+\h*\cosa+\h*\sina},
         {\cy+\h*\sina-\h*\cosa},1);

    \coordinate (C3) at
        ({\cx+\h*\cosa-\h*\sina},
         {\cy+\h*\sina+\h*\cosa},1);

    \coordinate (C4) at
        ({\cx-\h*\cosa-\h*\sina},
         {\cy-\h*\sina+\h*\cosa},1);

    \filldraw[
        fill=polyfill,
        fill opacity=.55,
        draw=polydraw,
        line width=.9pt
    ]
        (B4)--(B1)--(C1)--(C4)--cycle;

    \filldraw[
        fill=polyfill,
        fill opacity=.42,
        draw=polydraw,
        line width=.9pt
    ]
        (B1)--(B2)--(C2)--(C1)--cycle;

    \filldraw[
        fill=facetorange,
        fill opacity=.75,
        draw=polydraw,
        line width=.95pt
    ]
        (C1)--(C2)--(C3)--(C4)--cycle;

    \draw[polydraw,dashed,line width=.75pt]
        (B2)--(B3)--(B4);

    \draw[polydraw,dashed,line width=.75pt]
        (B3)--(C3);

    \coordinate (Oaxis) at (0,0,0);

    \draw[->,black!75,dashed,line width=.3pt]
        (Oaxis)--++(5,0,0)
        node[below right] {$x$};

    \draw[->,black!75,dashed,line width=.3pt]
        (Oaxis)--++(0,4.2,0)
        node[above right] {$y$};

    \draw[->,black!75,dashed,line width=.3pt]
        (Oaxis)--++(0,0,2)
        node[left] {$z$};
\end{tikzpicture}

\caption{The union $\Omega=Q\cup Q'$ of two unit cubes in
$\mathbb R^3$. The cube $Q$ is the standard unit cube, while
$Q'$ is a translate of a unit cube rotated by $45^\circ$ about
the vertical line through its center.}
\label{fig:two-cubes-one-rotated}
\end{figure}

We can also further extend Corollary \ref{cor-two-squares} to unions of three pairwise non-overlapping unit squares in \(\R^2\). We prove that Fuglede's conjecture holds for every such union, without assuming that the squares are parallel.

\begin{theorem}\label{thm:Fuglede-conjecture-3squares}
Let \(\Omega\subset\mathbb{R}^2\) be a union of three pairwise non-overlapping unit squares, not necessarily parallel to each other. Then \(\Omega\) is a spectral set if and only if \(\Omega\) tiles \(\mathbb{R}^2\) by translations.
\end{theorem}


The paper is organized as follows: In Section \ref{S2}, we prove Theorem \ref{th-Keller}, Keller's criterion for weak tiling measures. In Section \ref{S3}, we prove Theorem \ref{th-spectral-non-parallel} for the non-spectrality of non-parallel polytopes. In Section \ref{S4}, we prove Theorem \ref{th-two-cubes}. In Section \ref{S5}, we prove Theorem \ref{th-three-cubes-parallel} in $\R^1$. In Section \ref{S6}, we prove the reduction principle Theorem \ref{thm:layered-reduction}. In Section \ref{S7}, we provide a complete characterization of weak tiles that are unions of three intervals in $\R\times \Z$. In Section \ref{S8}, we prove Theorem \ref{th-three-cubes-parallel} in $\R^2$ and Theorem \ref{thm:Fuglede-conjecture-3squares}.

\section{Keller's criterion for weak tiling measures of the unit cube}\label{S2}
In the following, we present a complete proof of Theorem~\ref{th-Keller}. 
The proof is based on modifying the Fourier analytic proof for Keller's theorem discovered by Kolountzakis~\cite{Kol2000}. 

\subsection{Measures and distributions.}
By a \emph{measure} we mean a Borel measure on \(\mathbb R^d\). We write \(\supp(\mu)\) for the closed support of a measure \(\mu\), and \(\delta_\lambda\) for the Dirac measure at \(\lambda\). For a locally finite set \(\La\), we write \(\delta_\La:=\sum_{\lambda\in\La}\delta_\lambda\) for its counting measure. We denote by \(\m(A)\) the Lebesgue measure of a measurable set \(A\subset\mathbb R^d\). If needed, \(\m_k(A)\) denotes the \(k\)-dimensional Lebesgue measure of \(A\). 
We use the following terminology and standard facts as in~\cite[Section~2.2]{LevMat2022}.

If $\alpha$ is a tempered distribution on $\R^d$, and if $\varphi$ is a Schwartz function on $\R^d$, then we use $\langle \alpha,\varphi\rangle$ to denote the action of $\alpha$ on $\varphi$. A tempered distribution $\alpha$ is \emph{positive} if we have $\langle \alpha,\varphi\rangle \ge 0$ for any Schwartz function $\varphi \ge 0$. If a tempered distribution $\alpha$ is positive, then $\alpha$ is a positive measure~\cite[Theorem~2.1.7]{Hormander1990}. The Fourier transform $\widehat{\alpha}$ of a tempered distribution $\alpha$ is defined by
	\[
	\langle \widehat{\alpha}, \varphi\rangle := \langle \alpha, \widehat{\varphi}\rangle.
	\]

If $\mu$ is a measure on $\R^d$, then $\mu$ is said to be
\emph{locally finite} if we have $|\mu|(B) < \infty$ for every open ball $B$.

A measure $\mu$ on $\R^d$ is said to be \emph{translation-bounded} if for every (or equivalently, for some) open ball $B$ we have
	\[
	\sup_{x\in\R^d} |\mu|(B+x) < \infty.
	\]
If a measure $\mu$ on $\R^d$ is translation-bounded, then it is a tempered distribution.

A sequence of measures \(\{\mu_n\}\) on \(\mathbb R^d\) is said to be \emph{uniformly translation-bounded} if there exists a constant \(C>0\), independent of \(n\), such that for every $n$,
	\[
	\sup_{x\in\mathbb R^d}|\mu_n|(x+[0,1]^d)\le C.
	\]

We say that \(\mu_n\) converges \emph{vaguely} to a measure \(\mu\) if, for every continuous compactly supported function \(\varphi\), one has
	\[
	\int_{\mathbb R^d}\varphi\,d\mu_n
	\longrightarrow
	\int_{\mathbb R^d}\varphi\,d\mu .
	\]
	In this case, if \(\{\mu_n\}\) is uniformly translation-bounded, then the limit \(\mu\) is also translation-bounded. Moreover, for a uniformly translation-bounded sequence of measures \(\{\mu_n\}\), vague convergence is equivalent to convergence in the space of tempered distributions.

\subsection{Keller's theorem for weak tiling measures of cubes}
For a locally finite measure $\mu$, the {\it upper and lower Beurling densities} of $\mu$ are defined to be 
$$
D^{+}(\mu) = \limsup_{R\to\infty} \sup_{x\in \R^d} \frac{|\mu| (B(x,R))}{|B(x,R)|}, \ D^{-}(\mu) = \liminf_{R\to\infty} \inf_{x\in \R^d} \frac{|\mu| (B(x,R))}{|B(x,R)|}
$$
and $\mu$ is said to have a {\it uniform density} if $D^{+}(\mu) = D^{-}(\mu)$, in which case their common value is denoted by \(D(\mu)\) and is called the uniform density of \(\mu\). The definition of uniform density is unchanged if balls are replaced by cubes; in particular, the same density \(D\) is obtained with either choice of averaging sets.
We define the Fourier transform of $f\in L^1(\R^d)$ by 
$$
\widehat{f}(\xi) = \int f(x) e^{-2\pi i \xi\cdot x} ~dx.
$$
Throughout this section we take $Q:=(-1/2,1/2)^d$. In particular, 
$$
\widehat{{\mathbbm 1}_{{ Q}}}(\xi) = \prod_{i=1}^d \frac{\sin(\pi \xi_i)}{\pi \xi_i} 
$$ 
and thus 
\begin{equation}\label{zero_Q}
\{\xi\in\R^d: \widehat{{\mathbbm 1}_{Q}} (\xi) = 0\} = {\mathcal L}
\end{equation}


We first recall a fact in \cite{Gabardo2009}.
{
\begin{lemma}\cite[Lemma~4.5]{Gabardo2009}\label{lem:atom-existence}
Let \(\mu\) be a complex Borel measure on \(\mathbb{R}^m\) with the property that its total variation,
\(|\mu|\), is translation-bounded. Suppose that for some \(r>0\) and some \(\tau\in\mathbb{R}^m\),
\[
\operatorname{supp}(\widehat{\mu})\cap B_r(\tau)=\{\tau\},
\]
where
\[
B_r(\tau)=\{\xi\in\mathbb{R}^m:\ |\xi-\tau|<r\}.
\]
Then, there exists \(a\in\mathbb{C}\) such that
\[
\widehat{\mu}=a\delta_\tau \quad \text{on } B_r(\tau).
\]
\end{lemma}}

\begin{proposition}\label{prop-supp} The following assertions hold.

\begin{enumerate}[label=(\arabic*),font=\normalfont]
	\item Let $f\in L^1(\mathbb{R}^d)$ with $\int f dx\neq 0$, and let $\mu$ be a translation-bounded measure on $\mathbb{R}^d$ such that
		\[
		f * \mu = 1 \quad \text{a.e.}
		\]
		Then
		\begin{equation}\label{suppf}
		\supp(\widehat{\mu}) \subset \{0\}\cup\{\xi\in\mathbb{R}^d:\widehat f(\xi)=0\}.
		\end{equation}

    \item Conversely, suppose that \(0\leq f\in L^1(\mathbb R^d)\) and that \(\mu\) is a non-negative, translation-bounded measure with uniform density \(D(\mu)\). Assume that \eqref{suppf} holds. Suppose further that for every \(\varepsilon>0\) there exists a non-negative function \(f_\varepsilon\in L^1(\mathbb R^d)\) such that
		\begin{enumerate}[label=(\alph*),font=\normalfont]
			\item $\widehat{f_\varepsilon}\in C_c^\infty(\mathbb{R}^d)$;
			\item $\supp(\widehat{f_\varepsilon})\subset(\supp(\widehat{\mu})\setminus\{0\})^{\mathtt C}$;
			\item $\|f_\varepsilon-f\|_{L^1}\to0$.
		\end{enumerate}
		Then $f*\mu=c$ a.e., where $c=D(\mu)\int_{\mathbb R^d}f(x)\,dx\geq0.$
\end{enumerate}
\end{proposition}

\begin{proof}
	(1) This is well-known. See e.g.~\cite[Theorem~2.4]{KLM2025}.

	(2) Suppose that $f_\vep$ is as in the statement. If $f=0$ a.e., the conclusion is immediate. We may therefore assume that $\int f>0$. Then $\int f_\vep >0$ for all sufficiently small $\vep$. First we show that $(\int f_\vep)^{-1}f_\vep*\mu$ is a constant. Let $\varphi$ be a $C_c^{\infty}$ function. Then 
    \[
    \langle f_\varepsilon*\mu,\varphi\rangle
    =
    \left\langle \widehat{f_\varepsilon}\widehat{\mu},
    \mathcal F^{-1}\varphi\right\rangle
    =
    \left\langle \widehat{\mu},
    \widehat{f_\varepsilon}\mathcal F^{-1}\varphi\right\rangle,
    \]
    where $\mathcal F^{-1}$ denotes the inverse Fourier transform.
	But the function $\widehat{f_\vep}\mathcal{F}^{-1}\varphi$ is a $C_c^\infty$ function whose support intersects $\supp(\widehat{\mu})$ only at $0$. By \eqref{suppf} and Lemma~\ref{lem:atom-existence}, there exists \(a\in\mathbb C\) such that \(\widehat{\mu}=a\delta_0\) in a neighborhood of \(0\). Thus,
	\[
    \lrangle{f_\vep*\mu,\varphi}
    =
    \lrangle{\widehat{\mu},\widehat{f_\vep}\mathcal F^{-1}\varphi}
    =
    a\cdot\bigl(\widehat{f_\vep}\mathcal F^{-1}\varphi\bigr)(0)
    =
    a\cdot\int f_\vep\int\varphi,\]
	and, since this is true for an arbitrary $C_c^\infty$ function $\varphi$, we conclude that $f_\vep*\mu= a\cdot \int f_\vep.$

We next identify the constant \(a\). Averaging over a ball \(B_R\) and applying Tonelli's theorem, we obtain
    \[
    a\int_{\mathbb R^d} f_\varepsilon
    =
    \frac{1}{|B_R|}\int_{B_R}(f_\varepsilon*\mu)(x)\,dx
    =
    \int_{\mathbb R^d}
    f_\varepsilon(u)\frac{\mu(B_R-u)}{|B_R|}\,du.
    \]
    Since \(\mu\) has uniform density \(D(\mu)\), letting \(R\to\infty\) gives
    \[
    a\int_{\mathbb R^d} f_\varepsilon
    =
    D(\mu)\int_{\mathbb R^d} f_\varepsilon.
    \]
    Hence, if \(\int f_\varepsilon>0\), then
    \[
    a=D(\mu).
    \]

{For any ball \(B\subset\mathbb R^d\) and \(g\in L^1(\mathbb R^d)\), Tonelli's theorem and the change of variables \(u=x-y\) give
	\[
	\begin{aligned}
	\int_B |g*\mu(x)|\,dx
	&\leq
	\int_{\mathbb R^d}\int_B |g(x-y)|\,dx\,d|\mu|(y)\\
	&=
	\int_{\mathbb R^d}|g(u)|\,|\mu|(B-u)\,du\\
	&\leq
	C_{B,\mu}\|g\|_{L^1},
	\end{aligned}
	\]
	where $C_{B,\mu}:=\sup_{t\in\mathbb R^d}|\mu|(B-t)<\infty$.}  Applying this to $g=f-f_\vep$, we obtain that 
	$$f_\vep*\mu\to f*\mu,\quad \text{in } L^1(B).$$
	Since $B$ is arbitrary, we conclude that \[f*\mu=a \int f dx \quad a.e. \]
\end{proof}

\begin{proof}[Proof of Theorem \ref{th-Keller}.] Let $\nu = \delta_0+\mu$ be a locally finite measure such that ${\mathbbm 1}_{\mathcal Q} \ast \nu = 1$ a.e.  Then $\nu$ is translation-bounded and has uniform density $1$. By Proposition \ref{prop-supp} (1), we see that 
$$
\supp (\widehat{\nu}) \subset \{0\}\cup\{\xi\in\R^d: \widehat{{\mathbbm 1}_{Q}} (\xi) = 0\} = \{0\}\cup{\mathcal L}
$$
by (\ref{zero_Q}). In \cite[Page 596]{Kol2000}, it has been proved that if $f = |\widehat{{\mathbbm 1}_{Q}}|^2$, one can construct for each $\varepsilon>0$, $\widehat{f_{\varepsilon}}\in C^{\infty}_c$ such that $f_\vep\ge 0$, $\|f_{\varepsilon}-f\|_{L^1}\to 0$ and
$$
\supp \widehat{f_{\varepsilon}} \subset {Q}-{Q} = (-1,1)^d. 
$$
However,  $(-1,1)^d \cap {\mathcal L} = \varnothing$. We can guarantee all assumptions in Proposition \ref{prop-supp}(2) are satisfied. Hence, we have 
	$|\widehat{{\mathbbm 1}_{Q}}|^2\ast(\delta_0+\mu)= c$ a.e. for some constant $c$. Since \(\delta_0+\mu\) has uniform density \(1\), the level of the functional tiling is
		\[
			c=D(\nu)\int_{\mathbb R^d}
			\bigl|\widehat{\mathbbm 1_Q}(\xi)\bigr|^2\,d\xi
			=1.
		\]
    For every compact set $K\subset\mathbb R^d$, the decay of $|\widehat{\mathbbm 1}_{\mathcal Q}|^2$ gives
    \[
    	\sup_{\xi\in K}|\widehat{\mathbbm 1}_{\mathcal Q}|^2(\xi-y)
    	\leq C_K\prod_{j=1}^d(1+|y_j|)^{-2}.
    \]
    The right-hand side is integrable with respect to $\delta_0+\mu$
    by translation-boundedness. Since $f$ is continuous, dominated
    convergence shows that $|\widehat{\mathbbm 1}_{\mathcal Q}|^2*(\delta_0+\mu)$ is continuous.
    Hence the functional tiling identity holds everywhere, and in particular,
    \[
    	|\widehat{\mathbbm 1}_{\mathcal Q}(0)|^2
    	+\bigl(|\widehat{\mathbbm 1}_{\mathcal Q}|^2*\mu\bigr)(0)=1.
    \]
But $|\widehat{{\mathbbm 1}_{Q}}(0)|^2=1$. This forces that $|\widehat{{\mathbbm 1}_{ Q}}|^2 \ast\mu (0)  =0$. This means that $\mu$ has no support on $|\widehat{{\mathbbm 1}_{ Q}}(\xi)|^2>0$. Thus, $\supp (\mu)\subset \{\xi: \widehat{{\mathbbm 1}_{Q}}(\xi) = 0\} = {\mathcal L}$, completing the proof.
\end{proof}

In the next proposition we show that if a {non-negative}, locally finite Borel measure $\nu$ on $\mathbb{R}$ satisfies that the convolution of $\nu$ with the indicator function of a bounded interval is a constant almost everywhere, then $\nu$ is necessarily periodic with a period related to the length of the interval.

\begin{proposition}\label{prop:unit-interval-WTmeasure-1periodic}
Let $\nu$ be a non-negative, locally finite Borel measure on $\R$. Then there exists a constant $c\ge 0$ such that
\[
\mathbbm 1_{(0,1)}*\nu = c \qquad \text{a.e. on }\R
\]
if and only if $\nu$ is $1$-periodic, i.e.\ $\nu=\nu*\delta_1$.
\end{proposition}
\begin{proof}
Assume first that $\mathbbm{1}_{(0,1)}*\nu=c$ a.e. for some $c\ge 0$. Then $\nu$ is translation-bounded. If $c=0$, then we have immediately $\nu=0$. Now we suppose $c>0$. Taking Fourier transforms in the sense of tempered distributions, we obtain $\widehat{\mathbbm{1}_{(0,1)}}\cdot\widehat{\nu}=c\delta_0$. By Proposition~\ref{prop-supp} (1), we have
	\[\supp(\widehat{\nu})\subset \Z=\{0\}\cup\{\xi:\widehat{\mathbbm{1}_{(0,1)}}(\xi)=0\}.\]
	Since \(\widehat{\mathbbm{1}_{(0,1)}}\) has simple zeros precisely at the nonzero integers, the identity \(\widehat{\mathbbm{1}_{(0,1)}}\,\widehat{\nu}=c\delta_0\) rules out derivative terms in \(\widehat{\nu}\) at every \(k\in\mathbb Z\). Hence \(\widehat{\nu}\) is a pure point measure supported on \(\mathbb Z\) and we can write $\widehat{\nu}=\sum_{k\in\Z}a_k\delta_k$ for some $a_k\in\mathbb C$. For any $\varphi\in C_c^\infty(\R)$, we have 
	\[
    \begin{aligned}
    \langle \nu*\delta_1,\varphi\rangle
    &=\langle \nu,\varphi(\,\cdot+1)\rangle
    =\left\langle \widehat{\nu},e^{-2\pi i\xi}\mathcal F^{-1}\varphi(\xi)\right\rangle \\
    &=\sum_{k\in\mathbb Z}a_k e^{-2\pi i k}\mathcal F^{-1}\varphi(k)
    =\left\langle\widehat{\nu},\mathcal F^{-1}\varphi\right\rangle
    =\langle\nu,\varphi\rangle.
    \end{aligned}
    \]
	This implies that $\nu=\nu*\delta_1$.

Conversely, suppose that $\nu=\nu*\delta_1$. By periodicity, for any $x\in\R$, we have
    \[
    	\nu([x-1,x))=\nu([0,1))=:c.
    \]
Since $\nu$ has at most countably many atoms, we obtain
\[
	(\mathbbm 1_{(0,1)}*\nu)(x)
	=\nu((x-1,x))=c
	\qquad\text{for a.e. }x\in\mathbb R.
\]

\end{proof}

As an immediate consequence of Proposition~\ref{prop:unit-interval-WTmeasure-1periodic}, we have the following.

\begin{corollary}\label{cor:unit-interval-WTmeasure-constant-masses}
Let $\nu$ be a non-negative, locally finite Borel measure on $\R$ such that
	\[
	\mathbbm 1_{(0,1)}*\nu = c \qquad \text{a.e. on }\R
	\]
	for some $c\ge 0$. Then for any $a\in\R$, we have $$\nu(\{a+k\})=\nu(\{a\})\quad \text{for all } k\in\Z.$$
	
	Moreover, if $c=\nu(\{0\})=1$, then $\nu=\delta_{\Z}$.
\end{corollary}

\section{Non-spectrality of non-parallel polytopes} \label{S3}
 The goal of this section is to prove Theorem~\ref{th-spectral-non-parallel}. The proof of the theorem requires us to note that some results about the spectrality of convex polytopes can be extended to finite unions \cite{GL17,GL20}. First, we recall that (see \cite[Lemma 2.4]{GL17}), by the divergence theorem, a convex polytope $P$ satisfies the following:
$$
-2\pi i \xi \widehat{\mathbbm{1}_{P}} (\xi) = \sum_{F} n_F ~\widehat{\sigma_F} (\xi)
$$
where $\sigma_F$ is the surface measure on the facet $F$ and the sum is taken over all facets of $P$. Hence, by summing up all polytopes, the following also holds:
$$
-2\pi i \xi \widehat{\mathbbm{1}_{\Omega}} (\xi) = \sum_{F} n_F ~\widehat{\sigma_F} (\xi)
$$
where the sum is taken over all facets in $\Omega$. In the sequel, all estimates provided in \cite[Section 2]{GL17} continue to hold. In particular, Lemma 2.7 in \cite{GL17} holds in the following version.

\begin{lemma}\label{lem:asymptotic-facet}
Let $\Omega$ be a union of pairwise non-overlapping polytopes. Let ${\mathcal A}_{+},{\mathcal A}_{-}$ be the collections of all facets $F$ whose outward pointing normal is $e_1$ and $-e_1$, respectively. Then there exists $\alpha>0$ such that 
$$
-2\pi i \xi_1  \widehat{\mathbbm{1}_{\Omega}} (\xi)  = \sum_{A\in {\mathcal A}_{+}} \widehat{\sigma_A}(\xi)-\sum_{A\in {\mathcal A}_{-}} \widehat{\sigma_A}(\xi) + O(|\xi_1|^{-1}), \ \mbox{as}  ~|\xi_1|\to\infty
$$ 
in the cone
$$
K(\alpha) = \{\xi = (\xi_1,\cdots,\xi_d): |\xi_j| \le \alpha |\xi_1|, ~~\forall j>1\}.
$$
\end{lemma}

This implies the following lemma that we need. In the following, by a translation, we assume $Q = (-\frac12, \frac12)^d$. Let also $Q_{d-1} = (-\frac12, \frac12)^{d-1}.$ The following lemma is essentially \cite[Lemma 6.1]{GL17}.
\begin{lemma}\label{lem:asymptotic-facet-2}
Let $\Omega = Q\cup P_1\cup\cdots\cup P_N $ be a non-overlapping union of polytopes with the unit cube $Q$. Suppose that the only facets with normal vector $\pm e_1$ are the two corresponding facets of $Q$. Write $\xi = (\xi_1,\xi')$ where $\xi'\in \R^{d-1}$. Let $C>0$. Then there exists $D>0$ such that 
\begin{equation}\label{eq_asym}
\pi \xi_1  \widehat{\mathbbm{1}_{\Omega}} (\xi)  = \sin (\pi \xi_1)\widehat{\mathbbm 1}_{Q_{d-1}}(\xi') +O(|\xi_1|^{-1}), \ 
\end{equation}
for all  $|\xi_1|\ge D$ and $|\xi'| \le C$. 
\end{lemma}

\begin{proposition}\label{prop-containing-spectrum}
Let $\Omega = Q\cup P_1\cup\cdots\cup P_N $ be a union of non-overlapping polytopes with the unit cube $Q$. Suppose that for each $1\leq i\leq d$, the only facets of $\Om$ with normal vector $\pm e_i$ are the corresponding facets of $Q$. Suppose that $\Omega$ is spectral. Then $\Omega$ admits a spectrum $\Lambda$ such that 
$$
{\Lambda-\Lambda \subset {\mathcal L}\cup \{0\} }
$$
where ${\mathcal L}$ is as defined in \eqref{gridL}.
\end{proposition}

\begin{proof}
Let $\Gamma$ be a spectrum for $\Omega$ and let $\Gamma_k = \Gamma- (k,0,\cdots,0)$. Since all $\Gamma_k$ are spectra for $\Omega$ and they are uniformly discrete, we can find a subsequence $k_j$ such that $\Gamma_{k_j}$ converges weakly to a spectrum $\Lambda_1$ (see \cite[Section~3.3]{GL20}).

\medskip

{\bf Claim:} $\Lambda_1-\Lambda_1\subset (\Z\times \R^{d-1})\cup \left(\R\times {\mathcal Z}(\widehat{\mathbbm 1}_{Q_{d-1}})\right) = (\Z\times \R^{d-1})\cup \left(\R\times {\mathcal L}_{d-1}\right)=: H,$ 
\medskip

\noindent where ${\mathcal L}_{d-1}$ is the set of points in $\R^{d-1}$ having at least one coordinate that is a nonzero integer, which is also $\mathcal{Z}(\widehat{\mathbbm{1}}_{Q_{d-1}})$, i.e., the zero set of  $\widehat{\mathbbm 1}_{Q_{d-1}}$.

Assuming this claim, we can take the weak limits of translates of $\Lambda_1$ to obtain $\Lambda_2$ such that 
$$
\Lambda_2-\Lambda_2\subset \tau _2\left(H\right)
$$
where $\tau_2$ is the linear map that interchanges the first and the second coordinates, i.e. $\tau_2(x_1,x_2,\cdots,x_d)$ $ = (x_2,x_1,\cdots, x_d)$. Moreover, $\Lambda_2-\Lambda_2$ is in the closure of $\Lambda_1-\Lambda_1\subset H$, which is a closed set. Hence, $\Lambda_2-\Lambda_2\subset H\cap \tau_2(H)$. Continuing the process to all other coordinates, we obtain a spectrum $\Lambda$ such that 
$$
\Lambda-\Lambda \subset \bigcap_{j=1}^d\tau_j(H)
$$
where $\tau_j$ is the linear map that interchanges the first and $j^{th}$ coordinates,  and $\tau_1$ is the identity map. We notice that $\bigcap_{j=1}^d\tau_j(H)$ is exactly equal to ${\mathcal L}\cup \{0\}$ and this will complete the proof. 

\medskip
To justify the claim, we let $(u_1,v_1), (u_2,v_2)\in \R\times \R^{d-1}$ and they belong to $\Lambda_1$.   To show that $(u_1-u_2,v_1-v_2)\in H$, we assume that $v_1-v_2\not\in {\mathcal Z}(\widehat{\mathbbm 1}_{Q_{d-1}})$. The claim requires us to show that $u_1-u_2\in \Z$. The rest of the proof is basically the same as \cite[Lemma 6.2]{GL20}. We present the argument here for the sake of completeness.

We can  choose two subsequences $k'_j$ and $k_j''$ of $k_j$ such that they go to infinity and $k'_j-k_j''\to\infty$. Moreover, we can find sequences $(u'_j,v_j')$ and $(u_j'',v_j'')$ in $\Gamma$ such that 
$$
(u'_j-k'_j, v'_j)\to (u_1,v_1), \  (u''_j-k''_j, v''_j)\to (u_2,v_2).
$$
As $k'_j-k_j''\to\infty$, $u_j'-u_j''\to\infty$. This means that $(u'_j,v_j')$ and $(u''_{j},v_{j}'')$ must be distinct. Hence, 
$$
 \widehat{\mathbbm{1}_{\Omega}} (u_j'-u_j'',v_j'-v_j'') = 0. 
 $$   
 As $v_j'-v_j''$ converges, $|v_j'-v_j''|$ is bounded by some constant $C$. Using (\ref{eq_asym}), we must have
 $$
 \sin (\pi (u'_j-u_j''))\widehat{\mathbbm 1}_{Q_{d-1}}(v'_j-v_j'') \to 0 
 $$
 Our assumption that $v_1-v_2\notin {\mathcal Z}(\widehat{\mathbbm 1}_{Q_{d-1}})$ and the continuity of $\widehat{\mathbbm 1}_{Q_{d-1}}$ imply that there exists $c>0$ such that $|\widehat{\mathbbm 1}_{Q_{d-1}}(v_j'-v_j'')|\ge c$ for all sufficiently large $j$. This forces $\sin\bigl(\pi(u_j'-u_j'')\bigr)\to0.$ On the other hand, 
 $$
 |\sin (\pi (u'_j-u_j''))| = |\sin (\pi ((u'_j-k'_j)-(u_j''-k_j'')))|\to |\sin \pi (u_1-u_2)|.
 $$
 This implies that $\sin \pi (u_1-u_2) = 0$ and thus $u_1-u_2\in \Z$. This justifies the claim. 
\end{proof}

\begin{proof}[\bf Proof of Theorem \ref{th-spectral-non-parallel}.]
Suppose $\Omega$ is spectral. Proposition \ref{prop-containing-spectrum} shows us that there exists a spectrum $\Lambda$ such that 
	$$
	{\Lambda-\Lambda\subset {\mathcal L}\cup\{0\}}.
	$$   
However, this implies that $Q+(\lambda-\lambda')$ and $Q$  must be non-overlapping for distinct $\la,\la'\in\La$ and thus $Q+\Lambda$ is a packing of $\R^d$. However, this is a contradiction since the upper density of $\Lambda$ is now at most one, but $\Omega$ is a union of $Q$ with other polytopes, a spectrum must have a density equal to $|\Omega|>1$. The proof is complete. 
\end{proof}

\section{Unions of pairwise non-overlapping translates of the unit cube} \label{S4}

Having studied the weak tiling of the set $Q$, we can further establish a necessary condition for a finite union of pairwise non-overlapping translates of $Q$ to be a weak tile of $\R^d$. In particular, for the case where the union consists of two such cubes, we are able to deduce that the Fuglede conjecture holds true for this configuration. Moreover, in the one-dimensional setting, when considering the union of three unit intervals, we can provide a complete characterization for when this union is a weak tile of $\R$.

\subsection{Necessary condition for weak tiling.}
We first establish a necessary condition for a finite union of pairwise non-overlapping translates of \(Q\) to be a weak tile of \(\mathbb{R}^d\).

\begin{proposition}\label{prop-equal-size-necessary}
Let \(A = \{a_0=0, a_1, \dots, a_{n-1}\}\) be a finite subset of \(\mathbb{R}^d\) such that the translates \(a_j + Q\) are pairwise non-overlapping for all \(0 \le j \le n-1\). Define
\[
\Omega = A + Q = \bigcup_{j=0}^{n-1}(a_j + Q).
\]
If \(\Omega\) is a weak tile of \(\mathbb{R}^d\), then \(a_i \in \mathcal{L}\) for every \(i = 1, \dots, n-1\).

\end{proposition}

\begin{proof}
    Let $\nu$ be a non-negative, locally finite measure, and set $\mu:=\delta_0+\nu$. Suppose that
    \[
    \mathbbm{1}_{\Omega}\ast\mu=1 \qae
    \]
    As the translates of cubes are pairwise non-overlapping,  $ {\mathbbm 1}_{\Omega} = \left(\sum_{j=0}^{n-1} \delta_{a_j}\right)\ast {\mathbbm 1}_{Q}$. It follows immediately that 
    $$
      {\mathbbm 1}_{Q} \ast  \left(\sum_{j=0}^{n-1} \delta_{a_j}\right)\ast (\delta_0+\nu) = 1 \qae
    $$
    Let $\varrho = \sum_{j=1}^{n-1}\delta_{a_j} + \sum_{j=0}^{n-1}\delta_{a_j}\ast\nu$. The above implies that 
    $$
      {\mathbbm 1}_{Q} \ast (\delta_0+\varrho) = 1 \qae
    $$
    By Theorem \ref{th-Keller}, $\supp \left(\varrho\right) \subset {\mathcal L}$.  Finally, notice that $\varrho$ is a sum of {non-negative} measures. Then for all $r>0$ and $1\leq j\leq n-1$, the ball $B(a_j,r)$ must have $\varrho(B(a_j,r))\ge \delta_{a_j} (B(a_{j},r))>0$. Hence, $a_j\in \supp (\varrho)$. This implies that $a_j\in {\mathcal L}$.  
\end{proof}

We have the following corollary.
\begin{corollary}\label{prop-equal-size-necessary-translation}
	  Let \(A\) be a finite subset of \(\mathbb{R}^d\) such that the translates \(a + Q,\ a\in A\) are pairwise non-overlapping. 
    If \(\Omega=A + Q\) is a weak tile of \(\mathbb{R}^d\), then 
    \[
    A - A \subset \mathcal{L} \cup \{0\}.
    \]
\end{corollary}

For any point $x=(x_1,\cdots,x_d) \in \mathbb{R}^d$, we define the projections onto the coordinate axes by
	\[
	\pi_i : \mathbb{R}^d \to \mathbb{R}, \quad \pi_i(x) := x_i, 1\leq i\leq d.
	\]
	
	Then by Proposition~\ref{prop-equal-size-necessary}, we derive the following corollary for the two-dimensional setting, which asserts that the points in \(A\) must have integer coordinates in at least one canonical direction.

\begin{corollary}\label{cor-equal-size-necessary-dimension-2}
    Let $A = \{0 = a_0, a_1, \dots, a_{n-1}\} \subset \mathbb{R}^2$ be a finite set such that the translates $a_j + Q$ are pairwise non-overlapping.     If $\Omega=A+Q$ is a weak tile of  $\mathbb{R}^2$, then either
    \[
    \pi_1(a_i) \in \mathbb{Z} \quad \text{for all } 0 \le i \le n-1,
    \quad \text{or} \quad
    \pi_2(a_i) \in \mathbb{Z} \quad \text{for all } 0 \le i \le n-1.
    \]
    Moreover, if there exist distinct indices $i \neq j$ such that $\pi_1(a_i) = \pi_1(a_j)$, then
    \[
    \pi_2(a_i) - \pi_2(a_j) \in \mathbb{Z}^*,
    \]
    and similarly, if there exist distinct indices $i \neq j$ such that $\pi_2(a_i) = \pi_2(a_j)$, then
    \[
    \pi_1(a_i) - \pi_1(a_j) \in \mathbb{Z}^*.
    \]
\end{corollary}

\begin{proof}
 If $n\leq 2$, the conclusion is immediate. Now suppose that $n\ge 3$. We prove the argument by contradiction. Suppose that there exist $i,j$ such that both $\pi_1(a_i)\notin\Z$ and $\pi_2(a_j)\notin\Z$. If $i=j$, then $a_i\notin \mathcal{L}$, a contradiction to Proposition~\ref{prop-equal-size-necessary}. If $i\neq j$, then we translate $\Om$ by $-a_j$ and consider the new set $\Om' = \Om - a_j$. Then $0\in A' = A - a_j$ and by Proposition~\ref{prop-equal-size-necessary}, $a_i-a_j\in\mathcal L$, which implies either $\pi_1(a_i - a_j)\in\Z^*$ or $\pi_2(a_i - a_j)\in\Z^*$. If $\pi_1(a_i-a_j)\in\Z^*$, then $\pi_1(a_j)\notin \Z$ and $a_j\notin \mathcal L$, which contradicts Proposition~\ref{prop-equal-size-necessary}. The case $\pi_2(a_i - a_j)\in\Z^*$ can be treated similarly. This proves the first part of the corollary.
 
 The remaining part follows directly from Corollary \ref{prop-equal-size-necessary-translation}.
\end{proof}


Now we are ready to prove Theorem \ref{th-two-cubes}, which shows that Fuglede's conjecture holds for the union of two non-overlapping cubes of the same size.

\begin{proof}[Proof of Theorem \ref{th-two-cubes}.]
    Proposition \ref{prop-equal-size-necessary} shows that (1) implies (4). It suffices to show that (4) implies both (2) and (3). Then the proof will be complete since tiles and spectral sets are all weak tiles.

\smallskip

   \noindent {(4)$\Longrightarrow$(2).} By relabeling the coordinate axes and translating $\Om$ if necessary, we may assume without loss of generality that $a=(m,a_2,\ldots,a_d)$, where $m$ is a positive integer and $a_j\in\R$ for $j=2,\ldots,d$. In this case, let 
    $$
    {\mathcal J} = T\times \Z^{d-1}, \ \mbox{where} \  T = \{0,1,\cdots,m-1\}+ 2m\Z.
    $$
  Then ${\mathcal J}$ is a tiling set for $\Omega$. Indeed, 
  $$
  \Omega + \{(0,{\bf 0}),(1,{\bf 0}),\cdots, (m-1,{\bf 0})\} = R+ \{0,a\}\quad \text{a.e.}, \ \mbox{where}  \  R =(0,m)\times (0,1)^{d-1}, 
  $$
	${\bf 0}$ is the $(d-1)$-dimensional zero vector, and they clearly form a partition. Hence, we have 
$$
\Omega+  \{0,1,\cdots,m-1\}\times\Z^{d-1} = (0,2m)\times \R^{d-1}\qae
$$
and they continue to form a partition.  Finally,  translating the first coordinate by $2m\Z$ produces a tiling. This shows $\Omega$ is a tile. See Figure~\ref{fig:tiling-by-two-cubes}.
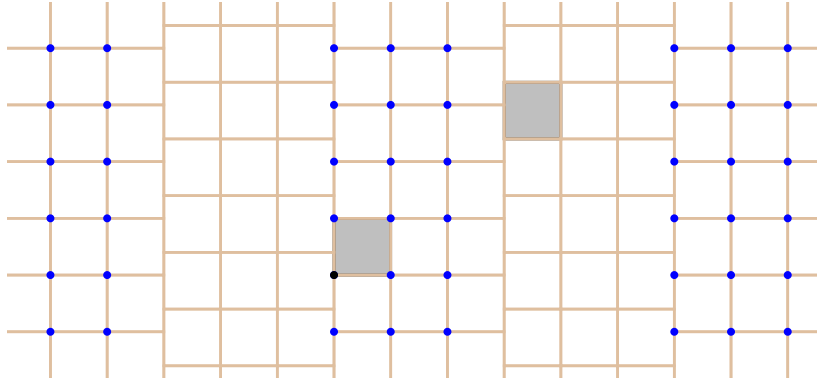
\begin{figure}[!ht]
	\centering
	\begin{tikzpicture}[scale=0.75]
		\draw[black,line width=1pt,fill=gray!50] (0,0) rectangle (1,1);
		\draw[black,line width=1pt,fill=gray!50] (3,2.4) rectangle (4,3.4);

		\newcommand{\DrawOm}[1]{
			\begin{scope}[shift={#1}]
				\draw[brown!50,line width=1pt] (0,0) rectangle (1,1);
				\draw[brown!50,line width=1pt] (3,2.4) rectangle (4,3.4);
			\end{scope}
		}

		\def\a{{1,0.8}}
		\def\b{{0,2}}

		\begin{scope}
			\clip (-5.75,-1.8) rectangle (8.5,4.8);
			\foreach \k in {-10,...,10}
			\foreach \r in {0,1,2}
			\foreach \n in {-10,...,10} {
			\pgfmathsetmacro\x{6*\k+\r}
			\pgfmathsetmacro\y{\n}
			\DrawOm{(\x,\y)}
			\fill[blue] (\x,\y) circle (2pt);
			}
		\end{scope}

		\fill[black] (0,0) circle (2pt);
	\end{tikzpicture}
	\caption{A tiling by translates of \(\Omega=\{(0,0),(3,3.4)\}+(0,1)^2\), in the case where the first coordinate of \(a=(3,3.4)\) is a nonzero integer. The gray region represents the set $\Om$, and the blue points represent the tiling set \(\mathcal J=(\{0,1,2\}+6\mathbb Z)\times\mathbb Z\).}
	\label{fig:tiling-by-two-cubes}
\end{figure}

\smallskip

     \noindent {(4)$\Longrightarrow$(3).}   Suppose that $a\in{\mathcal L}$ and we let $\gamma = \left(\frac1{2m},0,\cdots,0\right)$ and  define 
    $$
    \Lambda = \left\{0, \gamma\right\} + \Z^d.
    $$
Then $\Lambda$ is a spectrum for $\Omega$. To see this, we note that 
$$
\widehat{{\mathbbm 1}_{\Omega}}(\xi)  = (1+e^{-2\pi i a\cdot\xi}) \widehat{{\mathbbm 1}_{Q}}(\xi).
$$
We first verify orthogonality. Let $\lambda_1,\lambda_2\in\Lambda$ be distinct. If
$\lambda_1-\lambda_2\in\mathbb Z^d\setminus\{0\}$, then
$\widehat{\mathbbm 1}_{Q}(\lambda_1-\lambda_2)=0$. Otherwise, we can write $\lambda_1-\lambda_2=\pm\gamma+n$ for some $n=(n_1,n')\in\mathbb Z\times\mathbb Z^{d-1}.$ If $n'\neq0$, then again
$\widehat{\mathbbm 1}_{Q}(\pm\gamma+n)=0$. If $n'=0$, then $a\cdot(\pm\gamma+n)=\pm\frac12+mn_1$, so $1+e^{-2\pi i a\cdot(\pm\gamma+n)}=0$. Hence $\widehat{\mathbbm 1}_{\Omega}(\lambda_1-\lambda_2)=0$, and $\Lambda$ is mutually orthogonal.


Write $a=(m,a'),\ \xi=(\xi_1,\xi'),\ n=(n_1,n')$. Completeness can be checked via the tiling equation:
\[
\begin{aligned}
&\sum_{n\in\mathbb Z^d}
\sum_{\lambda\in\{0,\gamma\}}
\left|
\widehat{\mathbbm 1}_{\Omega}(\xi-\lambda-n)
\right|^2  \\
=&
\sum_{n'\in\mathbb Z^{d-1}}
\left|
\widehat{\mathbbm 1}_{Q_{d-1}}(\xi'-n')
\right|^2
\sum_{\lambda\in\{0,\gamma\}}
\left|
1+
e^{-2\pi i\left(m(\xi_1-\lambda_1)
+a'\cdot(\xi'-n')\right)}
\right|^2  \\
=&
4\sum_{n'\in\mathbb Z^{d-1}}
\left|
\widehat{\mathbbm 1}_{Q_{d-1}}(\xi'-n')
\right|^2
=4
=|\Omega|^2.
\end{aligned}
\]
\end{proof}

\section{Weak tilings of three unit intervals in \texorpdfstring{$\R$}{R}}\label{S5}
In the rest of the paper, we will be focusing on the equivalence of tiles, weak tiles and spectral sets for unions of three unit squares. We will need to study the corresponding problem on $\R$ and $\R\times \Z$.  In this section, we begin with an easy case study about the weak tiles in $\R$ consisting of three pairwise non-overlapping unit intervals. The one-dimensional case of Theorem~\ref{th-three-cubes-parallel} will be proved. Indeed, it follows directly from the next theorem.  By Proposition~\ref{prop-equal-size-necessary}, we may assume our set has the following form:
\begin{equation}\label{eq:three-intervals-R}
\Omega = (0,1) \cup (M,M+1) \cup (N,N+1),\quad M,N\in\Z,
\end{equation}
where $0$, $M$, and $N$ are distinct integers. Let $d:=\gcd(M,N)$.

\begin{theorem}\label{prop-three-intervals-R-tile-Z}
Let $\Omega$ be as in \eqref{eq:three-intervals-R}. The following are equivalent:
\begin{itemize}
    \item[\textup{(1)}] $\Om$ is a weak tile;
    \item[\textup{(2)}] $\Om$ is a tile;
    \item[\textup{(3)}] $\Om$ is a spectral set;
    \item[\textup{(4)}] $\{0,M/d,N/d\}\equiv \{0,1,2\}\pmod{3}$.
\end{itemize}
\end{theorem}

\begin{lemma}\label{three-intervals-R-WT-necessary}
	Let $\Om$ be as in \eqref{eq:three-intervals-R}. Suppose that $\Om$ weakly tiles $\R$ with a {non-negative}, locally finite measure $\nu$ such that $\nu(\{0\})=1$. Then 
	$$\nu(\{j\})+\nu(\{j-M\})+\nu(\{j-N\})=1,\quad \forall j\in\Z.$$
\end{lemma}
\begin{proof}
	Since $\mathbbm{1}_\Om*\nu=1$ a.e. and $\mathbbm{1}_{\Om}=\mathbbm{1}_{(0,1)}*\delta_{\{0,M,N\}}$, we have 
	$$\mathbbm{1}_{(0,1)}*(\delta_{\{0,M,N\}}*\nu) = 1 \qae$$
	By Theorem~\ref{th-Keller}, we have $\supp(\delta_{\{0,M,N\}}*\nu)\subset \Z$. Since 
	$$\supp(\delta_{\{0,M,N\}}*\nu)=\supp(\nu)+\{0,M,N\},$$
	we conclude that $\supp(\nu)\subset \Z$. Now we write $\nu=\sum_{j\in\Z}\nu(\{j\})\delta_j$.
	
	Since $\mathbbm{1}_{\Omega}*\nu=1$ a.e. and $\nu$ is nonnegative with $\nu(\{0\})=1$, we have $\nu(\{-M,-N\})=0$ and therefore $\delta_{\{0,M,N\}}*\nu$ has a unit mass at the origin. Applying Corollary \ref{cor:unit-interval-WTmeasure-constant-masses} to the measure $\delta_{\{0,M,N\}}*\nu$, we obtain
\[
\delta_{\{0,M,N\}}*\nu=\delta_{\mathbb{Z}}.
\]
Consequently,
\[
\nu(\{j\})+\nu(\{j-M\})+\nu(\{j-N\})=1,\quad \forall\, j\in\mathbb{Z}.
\]
		
%
\end{proof}

\begin{lemma}\label{lem:three-points-weak-tiling-measure}
	Let $\Om$ be as in \eqref{eq:three-intervals-R}. Suppose that $\Om$ weakly tiles $\R$ with a {non-negative}, locally finite measure $\nu$ such that $\nu(\{0\})=1$. Then 
	$$\nu(\{kM+\ell N\})=
		\begin{cases}
			1,& \text{if } k\equiv \ell\pmod{3},\\
			0,& \text{if } k\not\equiv \ell\pmod{3}.
		\end{cases}$$
\end{lemma}
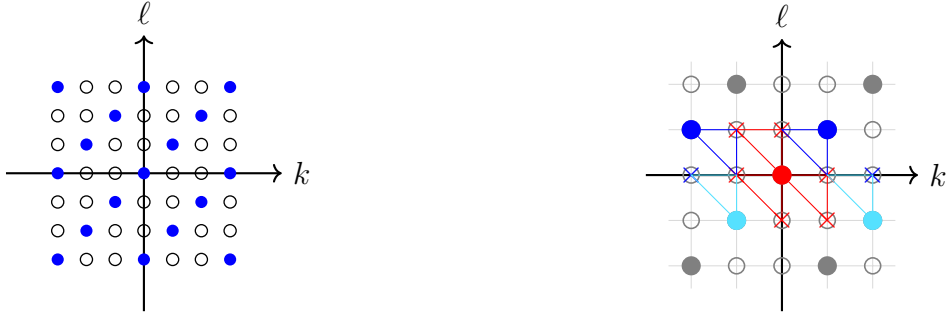
\begin{figure}[H]
	\centering

	\begin{subfigure}[t]{0.47\textwidth}
		\centering
		\begin{tikzpicture}[scale=0.38]
			\draw[->, thick] (-4.8,0)--(4.8,0) node[right] {$k$};
			\draw[->, thick] (0,-4.8)--(0,4.8) node[above] {$\ell$};

			\foreach \x in {-3,...,3}{
				\foreach \y in {-3,...,3}{
					\pgfmathtruncatemacro{\modx}{Mod(\x,3)}
					\pgfmathtruncatemacro{\mody}{Mod(\y,3)}
					\ifnum\modx=\mody
						\fill[blue] (\x,\y) circle (6pt);
					\else
						\draw[black, line width=0.5pt] (\x,\y) circle (6pt);
					\fi
				}
			}
		\end{tikzpicture}
	\end{subfigure}
	\hfill
	\begin{subfigure}[t]{0.47\textwidth}
		\centering
		\begin{tikzpicture}[scale=0.6]
			\draw[help lines, gray!30] (-2.5,-2.5) grid (2.5,2.5);

			\draw[->, thick] (-3,0)--(3,0) node[right] {$k$};
			\draw[->, thick] (0,-3)--(0,3) node[above] {$\ell$};

			\foreach \x in {-2,-1,...,2}{
				\foreach \y in {-2,-1,...,2}{
					\pgfmathtruncatemacro{\modx}{Mod(\x,3)}
					\pgfmathtruncatemacro{\mody}{Mod(\y,3)}
					\ifnum\modx=\mody
						\fill[gray] (\x,\y) circle (6pt);
					\else
						\draw[gray, line width=0.6pt] (\x,\y) circle (5pt);
					\fi
				}
			}

			\fill[blue] (1,1) circle (6pt);
			\fill[blue] (-2,1) circle (6pt);

			\fill[skyblue] (2,-1) circle (6pt);
			\fill[skyblue] (-1,-1) circle (6pt);

			\fill[red] (0,0) circle (6pt);

			\node[red] at (1,0) {$\times$};
			\node[red] at (1,-1) {$\times$};
			\node[red] at (0,1) {$\times$};
			\node[red] at (-1,1) {$\times$};
			\node[red] at (0,-1) {$\times$};
			\node[red] at (-1,0) {$\times$};

			\node[blue] at (-2,0) {$\times$};
			\node[blue] at (2,0) {$\times$};

			\draw[red] (-1,1)--(0,1)--(0,0)--cycle;
			\draw[red] (-1,0)--(0,0)--(0,-1)--cycle;
			\draw[red] (0,0)--(1,0)--(1,-1)--cycle;
			\draw[blue] (1,0)--(1,1)--(0,1)--cycle;
			\draw[blue] (-2,1)--(-1,1)--(-1,0)--cycle;
			\draw[skyblue] (-1,0)--(-2,0)--(-1,-1)--cycle;
			\draw[skyblue] (2,-1)--(2,0)--(1,0)--cycle;
		\end{tikzpicture}
	\end{subfigure}

	\caption{The periodic pattern of \(\nu(\{kM+\ell N\})\) and the corresponding propagation of the values \(s_{k,\ell}\).}
	\label{fig:nu-and-propagation}
\end{figure}

\begin{proof}
	For $k,\ell\in\mathbb Z$, set
	\[
		s_{k,\ell}:=\nu\bigl(\{kM+\ell N\}\bigr).
	\]
	Since $\nu$ is nonnegative, we have
	$s_{k,\ell}\geq0$ for all $k,\ell\in\mathbb Z$. By
	Lemma~\ref{three-intervals-R-WT-necessary},
	\begin{equation}\label{eq:three-intervals-R-WT-necessary}
		s_{k,\ell}+s_{k-1,\ell}+s_{k,\ell-1}=1,
		\qquad k,\ell\in\mathbb Z.
	\end{equation}
	Thus, for every $(k,\ell)\in\mathbb Z^2$, the values at the three
	vertices $(k,\ell)$, $(k-1,\ell)$, and $(k,\ell-1)$ have sum $1$.
	In particular, whenever one of $s_{k,\ell}$, $s_{k-1,\ell}$ and $s_{k,\ell-1}$ is $1$, the other
	two must be $0$; conversely, whenever two of them are $0$, the
	remaining one must be $1$.

Let
\[
	L:=\bigl\{(k,\ell)\in\mathbb Z^2:k\equiv\ell\pmod 3\bigr\}.
\]
Then $L$ is generated by $(1,1)$ and $(-2,1)$. Indeed, if
$k\equiv\ell\pmod 3$, then
\[
	(k,\ell)
	=
	\frac{k+2\ell}{3}(1,1)
	+
	\frac{\ell-k}{3}(-2,1),
\]
and both coefficients are integers.

We first establish a local propagation property. Suppose that
$s_{i,j}=1$ for some $i,j\in\mathbb Z$. Applying
\eqref{eq:three-intervals-R-WT-necessary} at
$(i,j)$, $(i+1,j)$, and $(i,j+1)$, respectively, gives
\[
\begin{aligned}
	s_{i,j}+s_{i-1,j}+s_{i,j-1}&=1,\\
	s_{i+1,j}+s_{i,j}+s_{i+1,j-1}&=1,\\
	s_{i,j+1}+s_{i-1,j+1}+s_{i,j}&=1.
\end{aligned}
\]
Since $s_{i,j}=1$ and the sequence is nonnegative, it follows that
\begin{equation}\label{eq:three-intervals-zero-neighbors}
	0
	=s_{i-1,j}
	=s_{i,j-1}
	=s_{i+1,j}
	=s_{i+1,j-1}
	=s_{i,j+1}
	=s_{i-1,j+1}.
\end{equation}
Applying \eqref{eq:three-intervals-R-WT-necessary} at
$(i+1,j+1)$, $(i-1,j+1)$, and $(i+1,j-1)$, and using
\eqref{eq:three-intervals-zero-neighbors}, we obtain
\[
	s_{i+1,j+1}
	=s_{i-2,j+1}
	=s_{i+1,j-2}
	=1.
\]
Since $i,j$ are arbitrary, the same conclusion may be applied again
at any point where the value of the sequence is $1$. Applying it at
$(i-2,j+1)$ and $(i+1,j-2)$ gives, respectively,
\[
	s_{i-1,j-1}=1,
	\qquad
	s_{i+2,j-1}=1.
\]
Consequently, we obtain the bidirectional propagation rule
\begin{equation}\label{eq:three-intervals-bidirectional-propagation}
	s_{i,j}=1
	\quad\Longrightarrow\quad
	s_{i+1,j+1}
	=s_{i-1,j-1}
	=s_{i-2,j+1}
	=s_{i+2,j-1}
	=1.
\end{equation}

By assumption, $s_{0,0}=\nu(\{0\})=1$. Starting from $(0,0)$ and repeatedly applying \eqref{eq:three-intervals-bidirectional-propagation}, we may move in both directions along $(1,1)$ and $(-2,1)$. Since these two directions generate $L$, it follows that $s_{k,\ell}=1$ for every $(k,\ell)\in L$.

It remains to determine the values outside $L$. Let $(k,\ell)\notin L$. If $k-\ell\equiv1\pmod 3$, then $(k-1,\ell)\in L$, and hence $s_{k-1,\ell}=1$. Applying
\eqref{eq:three-intervals-R-WT-necessary} at $(k,\ell)$ gives $s_{k,\ell}+1+s_{k,\ell-1}=1$. By nonnegativity, $s_{k,\ell}=0$. Similarly, if $k-\ell\equiv-1\pmod 3$, then $(k,\ell-1)\in L$, so $s_{k,\ell-1}=1$, and
\eqref{eq:three-intervals-R-WT-necessary} again yields
$s_{k,\ell}=0$. The propagation described above is illustrated
in Figure~\ref{fig:nu-and-propagation}.

Therefore,
\[
	s_{k,\ell}
	=
	\begin{cases}
		1,& k\equiv\ell\pmod 3,\\
		0,& k\not\equiv\ell\pmod 3.
	\end{cases}
\]
\end{proof}

\begin{lemma}\label{lem:MN-mod-3-generate-gcd}
Let $M,N\in\mathbb{Z}^*$ be two distinct integers and let $d=\gcd(M,N)$. Assume that
	$$\left\{0,\frac{M}{d},\frac{N}{d}\right\}\not\equiv \{0,1,2\}\pmod{3}.$$
Then the set
	\[
	H := \{\,kM + \ell N : k,\ell\in\mathbb{Z},\ k\equiv \ell \pmod{3}\,\}
	\]
	coincides with the subgroup $d\mathbb{Z}\subset\mathbb{Z}$; that is, $H = d\mathbb{Z}$.
\end{lemma}
\begin{proof}
Write $M=dA$ and $N=dB$, where $\gcd(A,B)=1$.
	Then $H=dH'$, where
	\[
	H':=\{kA+\ell B:\ k,\ell\in\Z,\ k\equiv \ell \pmod 3\},
	\]
	so it suffices to show that \(H'=\Z\).

Writing \(\ell=k+3t\), we get $kA+\ell B=k(A+B)+3tB$,
	and hence
	\[
	H'=(A+B)\Z+3B\Z=\gcd(A+B,3B)\Z.
	\]
	Since \(\gcd(A,B)=1\), we have \(\gcd(A+B,B)=1\), so \(\gcd(A+B,3B)\in\{1,3\}\).

	If \(\gcd(A+B,3B)=3\), then \(3\mid A+B\). As \(\gcd(A,B)=1\), this forces \(3\nmid A\) and \(3\nmid B\), so
	\[
	\{0,A,B\}\equiv \{0,1,2\}\pmod 3,
	\]
	contrary to assumption. Therefore \(\gcd(A+B,3B)=1\), and thus \(H'=\Z\).
\end{proof}

Now we are ready to prove Theorem~\ref{prop-three-intervals-R-tile-Z}.
\begin{proof}[Proof of Theorem~\ref{prop-three-intervals-R-tile-Z}]
	If $\Om$ weakly tiles $\R$ with a {non-negative}, locally finite measure $\nu$ such that $\nu(\{0\})=1$, then by 
    Lemma~\ref{lem:three-points-weak-tiling-measure}, we have
	$$\nu(\{kM+\ell N\})=1,\quad  \text{if } k\equiv \ell\pmod{3}.$$

	Our proof is by contradiction. Suppose that the set $\{0, M/d, N/d\}$ does not form a complete residue system modulo $3$. Then by Lemma~\ref{lem:MN-mod-3-generate-gcd}, we have
	$$d\Z=\{kM+\ell N: k,\ell\in\Z, k\equiv \ell\pmod{3}\}.$$
	In particular, we have $M\in d\Z$, hence $\nu(\{M\})=1$.
	However, by Lemma~\ref{lem:three-points-weak-tiling-measure}, we have $\nu(\{M\})=0$,
	which is a contradiction. Therefore, $\{0, M/d, N/d\}$ forms a complete residue system modulo $3$. This proves that (1) implies (4).
    Conversely, if (4) holds, then $T=\{0,1,\ldots,d-1\}+3d\Z$ is a tiling set for $\Omega$, and $\La=\{0,1/(3d),2/(3d)\}+\Z$ is a spectrum for $\Omega$. Since both (2) and (3) imply (1), the proof is complete.
    
\end{proof}

\section{Reduction from \texorpdfstring{$\R^2$}{R2} to \texorpdfstring{$\R\times\Z$}{RtimesZ}}\label{S6}

In this section, we prove Theorem~\ref{thm:layered-reduction}. Throughout the section, \(S\subset\mathbb Z\) is finite, and we consider the associated layered sets
	\[
	\Omega
	=
	\bigcup_{j\in S}\Omega_j\times(j,j+1)
	\subset\mathbb R^2,
	\qquad
	\widetilde{\Omega}
	=
	\bigcup_{j\in S}\Omega_j\times\{j\}
	\subset\mathbb R\times\mathbb Z,
	\]
where each \(\Omega_j\subset\mathbb R\) is bounded and measurable with finite positive Lebesgue measure. As illustrated in Figure~\ref{fig:layered-correspondence}, each horizontal strip \(\Omega_j\times(j,j+1)\) is collapsed onto the integer level \(\Omega_j\times\{j\}\). The theorem asserts that \(\Omega\) is a tile, a weak tile, or a spectral set if and only if \(\widetilde{\Omega}\) has the corresponding property. 

\begin{figure}[H]
\centering
\begin{tikzpicture}[scale=1.15]

\begin{scope}[xshift=0cm]
    \node[font=\small] at (2.1,4.35) {$\Omega\subset\mathbb R^2$};

    \foreach \y/\lab in {0/0,1/1,2/2,3/3}{
        \draw[gray!35, thin] (-0.2,\y) -- (4.4,\y);
        \node[left, gray, font=\scriptsize] at (-0.25,\y) {$\lab$};
    }
    \draw[gray!35, thin] (-0.2,4) -- (4.4,4);
    \node[left, gray, font=\scriptsize] at (-0.25,4) {$4$};

    \fill[gray!35] (0.4,0) rectangle (1.5,1);
    \fill[gray!35] (2.0,1) rectangle (3.4,2);
    \fill[gray!35] (1.0,3) rectangle (2.7,4);

    \draw[line width=0.8pt] (0.4,0) rectangle (1.5,1);
    \draw[line width=0.8pt] (2.0,1) rectangle (3.4,2);
    \draw[line width=0.8pt] (1.0,3) rectangle (2.7,4);


    \node[font=\scriptsize, gray] at (4.15,-0.25) {$x$};
    \node[font=\scriptsize, gray] at (-0.55,4.4) {$y$};
\end{scope}

\draw[->, line width=0.8pt] (4.9,2) -- (6.0,2);
\node[font=\scriptsize] at (5.45,2.35) {collapse};
\node[font=\scriptsize] at (5.45,1.65) {each strip};

\begin{scope}[xshift=6.5cm]
    \node[font=\small] at (2.1,4.35) {$\widetilde{\Omega}\subset\mathbb R\times\mathbb Z$};

    \foreach \y/\lab in {0/0,1/1,2/2,3/3,4/4}{
        \draw[gray!35, thin] (-0.2,\y) -- (4.4,\y);
        \node[right, gray, font=\scriptsize] at (4.45,\y) {$\mathbb R\times\{\lab\}$};
    }

    \draw[line width=3pt] (0.4,0) -- (1.5,0);
    \draw[line width=3pt] (2.0,1) -- (3.4,1);
    \draw[line width=3pt] (1.0,3) -- (2.7,3);

\end{scope}

\end{tikzpicture}
\caption{The geometric correspondence between \(\Omega\subset\mathbb R^2\) and \(\widetilde{\Omega}\subset\mathbb R\times\mathbb Z\): each horizontal strip \(\Omega_j\times(j,j+1)\) is collapsed onto the integer level \(\Omega_j\times\{j\}\).}
\label{fig:layered-correspondence}
\end{figure}
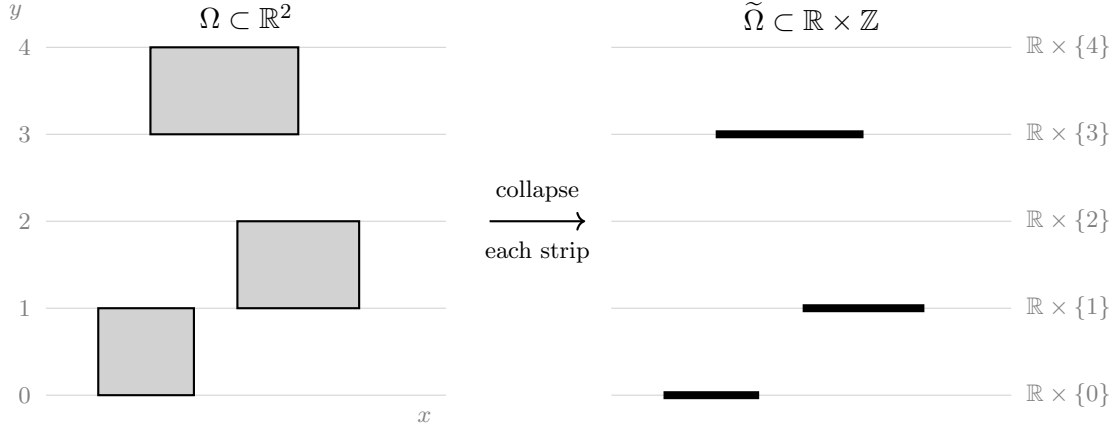
On $\R\times\Z$, we use the product of Lebesgue measure and counting measure as Haar measure. Its dual group is identified with $\R\times\mathbb T$, where $\mathbb T=\R/\Z$, with characters
\[
(x,j)\longmapsto e^{2\pi i(\lambda x+\eta j)},
\qquad (\lambda,\eta)\in\R\times\mathbb T.
\]
Tiling, weak tiling, and spectrality on $\R\times\Z$ are understood with respect to this Haar measure and these characters. In the definition of a weak tile, the weak tiling measure $\mu$ is required to have mass $1$ at the origin, that is, $\mu(\{(0,0)\})=1$.



\subsection{Reduction of tiling from \(\mathbb R^2\) to \(\mathbb R\times\mathbb Z\)} For translational tilings, we fix a generic horizontal phase and identify the sections of the translates of \(\Omega\) along the lines \(y=m+t\), \(m\in\mathbb Z\), with translates of \(\widetilde{\Omega}\).
\begin{proposition}\label{prop:tiling-layered-reduction}
Let \(S\subset\mathbb Z\) be finite, and let
	\[
	\Omega
	=
	\bigcup_{j\in S}\Omega_j\times(j,j+1)
	\subset\mathbb R^2,
	\qquad
	\widetilde{\Omega}
	=
	\bigcup_{j\in S}\Omega_j\times\{j\}
	\subset\mathbb R\times\mathbb Z,
	\]
	where the sets \(\Omega_j\subset\mathbb R\) are bounded and measurable. Then \(\Omega\) is a tile in \(\mathbb R^2\) if and only if \(\widetilde{\Omega}\) is a tile in \(\mathbb R\times\mathbb Z\).
\end{proposition}

\begin{proof}
Suppose that \(T\subset\mathbb R^2\) is a tiling set for \(\Omega\). Thus
	\begin{equation}\label{eq:tiling}
	\sum_{(u,v)\in T}
	\mathbbm 1_{\Omega}(x-u,y-v)
	=
	1
	\qquad\text{for a.e. }(x,y)\in\mathbb R^2.
	\end{equation}
	Let $N$ be the set for which (\ref{eq:tiling}) fails and $N_y = \{x\in\R: (x,y)\in N\}$. By Fubini's theorem, \(N_y\) has one-dimensional Lebesgue measure zero for almost every \(y\). Hence there exists a full-measure set \(G \subset \mathbb{R}\) such that, for every \(y \in G\), \eqref{eq:tiling} holds for almost every \(x \in \mathbb{R}\). Consider 
    $$
    {\mathcal N}: = \{t\in(0,1): t+m\not\in G ~ {\textup{for some }} m\in\Z\}\cup \{t\in(0,1): v-t{\in}\Z ~ \textup{for some } (u,v)\in T\}.
    $$
    As $\Z$ and $T$ are countable, ${\mathcal N}$ is of measure zero. We may choose \(t\in(0,1)\) in the complement of ${\mathcal N}$ such that, for every
	\(m\in\mathbb Z\),
	\[
	\sum_{(u,v)\in T}
	\mathbbm 1_{\Omega}(x-u,m+t-v)
	=
	1
	\qquad\text{for a.e. }x\in\mathbb R,
	\]
 and $v-t\notin\mathbb Z\ \text{for every }(u,v)\in T$.

For every \(m\in\mathbb Z\) and \(s\in(0,1)\), we have
\begin{equation}\label{eq:layered-section-identity}
	\mathbbm 1_{\Omega}(x,m+s)
	=
	\mathbbm 1_{\widetilde{\Omega}}(x,m).
\end{equation}
Indeed, since \(m+s\in(j,j+1)\) if and only if \(m=j\), we have
\[
\begin{aligned}
	\mathbbm 1_{\Omega}(x,m+s)
	=
	\sum_{j\in S}
	\mathbbm 1_{\Omega_j}(x)
	\mathbbm 1_{(j,j+1)}(m+s)
	=
	\sum_{j\in S}
	\mathbbm 1_{\Omega_j}(x)
	\mathbbm 1_{\{j\}}(m)
	=
	\mathbbm 1_{\widetilde{\Omega}}(x,m).
\end{aligned}
\]

For every \((u,v)\in T\), since \(v-t\notin\mathbb Z\), we have
\[
	0<k_t(v)+t-v<1,
	\qquad
	\text{where } k_t(v):=\lceil v-t\rceil\in\mathbb Z.
\]
Thus $m+t-v= m-k_t(v)+\bigl(k_t(v)+t-v\bigr)$.
Applying \eqref{eq:layered-section-identity} with
\(m-k_t(v)\) in place of \(m\) and \(k_t(v)+t-v\) in place of \(s\),
we obtain
\begin{equation*}\label{eq:layered-section-shift-identity}
	\mathbbm 1_{\Omega}(x,m+t-v)
	=
	\mathbbm 1_{\widetilde{\Omega}}
	\bigl(x,m-k_t(v)\bigr).
\end{equation*}

Define an integer-valued measure on \(\mathbb R\times\mathbb Z\) by \[ \eta_t:= \sum_{(u,v)\in T}\delta_{(u,k_t(v))}.\] For every \(m\in\mathbb Z\), we now obtain 
\[ \begin{aligned} (\mathbbm 1_{\widetilde{\Omega}}*\eta_t)(x,m) &= \sum_{(u,v)\in T}\mathbbm{1}_{\widetilde{\Om}}(x-u,m-k_t(v)) \\ &=  \sum_{(u,v)\in T} \mathbbm 1_{\Omega}(x-u,m+t-v)\\ &=1\quad \text{a.e. }x\in\R,\end{aligned} \]
where the second equality follows from \eqref{eq:layered-section-identity}.
Therefore $\mathbbm 1_{\widetilde{\Omega}}*\eta_t=1 \ \text{a.e. on }\mathbb R\times\mathbb Z$.

It remains only to show that \(\eta_t\) has no multiple atoms. Suppose that \(\eta_t(\{(u,k)\})\geq2\) for some \((u,k)\). Then we have 
	\[
	1=\mathbbm 1_{\widetilde{\Omega}}*\eta_t\ge \mathbbm 1_{\widetilde{\Omega}}*(2\delta_{(u,k)})= 2
	\qquad\text{a.e. on }\widetilde{\Omega}+(u,k),
	\]
	a contradiction. Hence $\eta_t=\delta_{T_t}$ for some set $T_t\subset\mathbb R\times\mathbb Z$. Since $\delta_T$ is locally finite and the map $(u,v)\mapsto (u,\lceil v-t\rceil)$ changes only the second coordinate by less than one unit, $\eta_t$ is also locally finite. In particular, $T_t$ is countable. Thus \(\widetilde{\Omega}\) is a tile of \(\mathbb R\times\mathbb Z\) with the tiling set $T_t$.

We omit the proof of the converse as it follows by taking the weak tiling measure $\nu$ as $\sum_{t\in T} \delta_t$ in Proposition \ref{prop:weak-tiling-layered-reduction} in the next subsection.
\end{proof}

\subsection{Reduction of weak tilings from \texorpdfstring{$\R^2$}{R2} to \texorpdfstring{$\R\times\Z$}{RtimesZ}} For weak tilings, we average these sectional reductions over all phases; equivalently, a mass lying between two adjacent integer levels is distributed between them with the corresponding linear weights.
\begin{proposition}\label{prop:weak-tiling-layered-reduction}
Let \(S\subset\mathbb Z\) be finite, and let
	\[
	\Omega
	=
	\bigcup_{j\in S}\Omega_j\times(j,j+1)
	\subset\mathbb R^2,
	\qquad
	\widetilde{\Omega}
	=
	\bigcup_{j\in S}\Omega_j\times\{j\}
	\subset\mathbb R\times\mathbb Z,
	\]
	where each \(\Omega_j\subset\mathbb R\) is bounded and measurable.
	Then \(\Omega\) is a weak tile in \(\mathbb R^2\) if and only if \(\widetilde{\Omega}\) is a weak tile in \(\mathbb R\times\mathbb Z\).
\end{proposition}

\begin{proof}
Suppose that \(\widetilde{\Omega}\) is a weak tile in \(\mathbb R\times\mathbb Z\). Let \(\nu\) be a nonnegative locally finite measure on \(\mathbb R\times\mathbb Z\) such that $\nu(\{(0,0)\})=1$ and $$\mathbbm{1}_{\widetilde{\Omega}}*\nu=1 \quad \text{a.e. on }\mathbb R\times\mathbb Z.$$
	We regard \(\nu\) as a measure on \(\mathbb R^2\), supported on \(\mathbb R\times\mathbb Z\). For \(m\in\mathbb Z\) and \(t\in(0,1)\), we have   \[ \begin{aligned} (\mathbbm{1}_{\Omega}*\nu)(x,m+t) &= \int_{\mathbb R\times\mathbb Z} \mathbbm{1}_{\Omega}(x-u,m+t-k)\,d\nu(u,k)\\  &= \int_{\mathbb R\times\mathbb Z} \mathbbm{1}_{\widetilde{\Om}}(x-u,m-k)\,d\nu(u,k)\\ &= (\mathbbm{1}_{\widetilde{\Om}}*\nu)(x,m)\\ &=1 \qae x\in\R,\end{aligned} \]
    where the second equality follows from \eqref{eq:layered-section-identity}. Therefore $\mathbbm{1}_{\Omega}*\nu=1 \ \text{a.e. on }\mathbb R^2$, and hence \(\Omega\) is a weak tile in \(\mathbb R^2\). It remains to prove the converse, which is the harder part of the proof.

Assume now that \(\Omega\) is a weak tile in \(\mathbb R^2\). Then there exists a nonnegative locally finite measure \(\mu\) on
	\(\mathbb R^2\) such that $\mu(\{(0,0)\})=1$ and 
	\[
	\mathbbm{1}_{\Omega}*\mu=1
	\qquad\text{a.e. on }\mathbb R^2.
	\]
For every \(k\in\mathbb Z\), define a nonnegative measure \(\nu_k\)  on \(\mathbb R\) by
	\[
	\begin{aligned}
	\nu_k(E)
	:={}
	\int_{E\times[k,k+1)}
	(k+1-v)\,d\mu(u,v)
	+
	\int_{E\times[k-1,k)}
	(v-k+1)\,d\mu(u,v)
	\end{aligned}
	\]
	for every Borel set \(E\subset\mathbb R\). Define a measure on \(\mathbb R\times\mathbb Z\) by
	\[
	\nu:=\sum_{k\in\mathbb Z}\nu_k\times\delta_k.
	\]
\begin{figure}[H]
	\centering
	\begin{tikzpicture}[
		scale=0.8,
		>=Stealth,
		every node/.style={font=\normalsize},
		axis/.style={->, line width=0.9pt},
		guide/.style={gray!60, dashed, line width=0.7pt},
		bluearr/.style={blue!80!black, ->, line width=1.1pt},
		orangearr/.style={orange!90!red, ->, line width=1.1pt}
	]

	\node[font=\large] at (2.4,6.0) {$\mathbb{R}^2$};

	\draw[axis] (-0.2,0) -- (5.2,0) node[below right] {$x$};
	\draw[axis] (0,-0.2) -- (0,5.5) node[above left] {$v$};

	\draw[guide] (0,0.5) -- (4.9,0.5);
	\draw[guide] (0,2.5) -- (4.9,2.5);
	\draw[guide] (0,4.5) -- (4.9,4.5);

	\node[left] at (0,4.5) {$v=k+1$};
	\node[left] at (0,2.5) {$v=k$};
	\node[left] at (0,0.5) {$v=k-1$};

	\draw[guide] (2.7,0) -- (2.7,4.5);

	\fill (2.7,3.0) circle (2.8pt);
	\node[below] at (2.7,0) {$u$};

	\node[left] at (2.35,3.15) {$[k,k+1)$};
	\node[left] at (2.35,1.35) {$[k-1,k)$};

	\node[above right] at (1.4,3.5) {mass at $(u,k+s)$};


	\node[font=\large] at (10.4,6.0) {$\mathbb{R}\times\mathbb{Z}$};

	\draw[axis] (7.8,0) -- (13.0,0);

	\draw[guide] (10.2,0) -- (10.2,4.2);

	\draw[axis] (10.2,4.5) -- (12.0,4.5);
	\draw[axis] (10.2,2.5) -- (12.0,2.5);

	\fill[blue!80!black] (10.2,4.5) circle (2.8pt);
	\fill[orange!90!red] (10.2,2.5) circle (2.8pt);

	\node[above right] at (10.2,4.5) {$(u,k+1)$};
	\node[above right] at (10.2,2.5) {$(u,k)$};

	\node[right] at (12.15,4.5) {$\mathbb{R}\times\{k+1\}$};
	\node[right] at (12.15,2.5) {$\mathbb{R}\times\{k\}$};

	\node[below] at (10.2,0) {$u$};

	\draw[bluearr]
		(2.7,3.0)
		to[out=8,in=180]
		node[pos=0.62, above, text=blue!80!black] {weight $s$}
		(10.1,4.5);

	\draw[orangearr]
		(2.7,3.0)
		to[out=-18,in=180]
		node[pos=0.58, below, text=orange!90!red] {weight $1-s$}
		(10.1,2.5);

	\node at (7,-1.5)
	{\(
	\delta_{(u,k+s)}
	\longmapsto
	(1-s)\,\delta_{(u,k)}
	+
	s\,\delta_{(u,k+1)}
	\)};

	\end{tikzpicture}
	\caption{Compression of a point mass from $\mathbb{R}^2$ to $\mathbb{R}\times\mathbb{Z}$. 
	A mass located at height $k+s$ is split between the two adjacent integer levels $k$ and $k+1$, with weights $1-s$ and $s$, respectively.}
	\label{fig:compression-point-mass}
\end{figure}
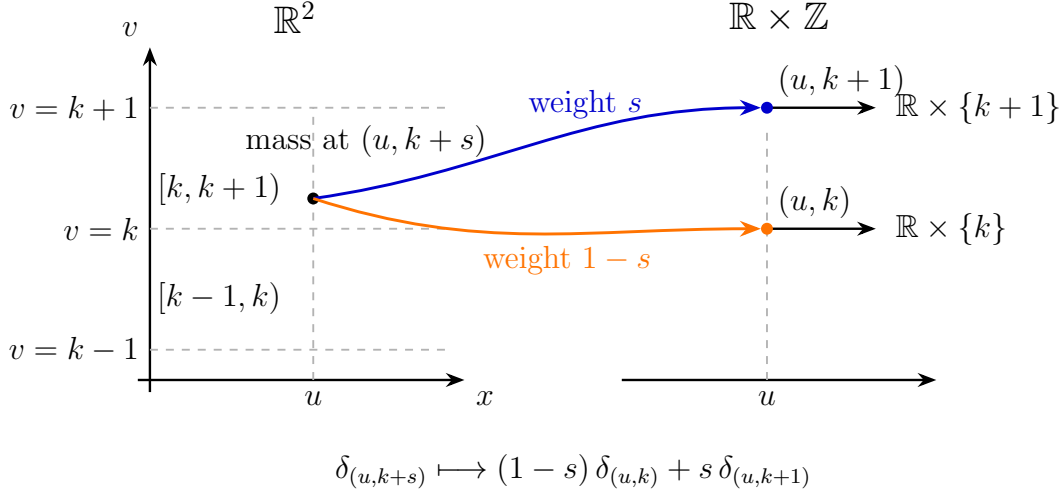
    
\noindent Geometrically, if \(v=k+s\), where \(k\in\mathbb Z\) and \(s\in[0,1)\), then the mass of \(\mu\) at height \(v\) is divided between the two adjacent integer levels \(k\) and \(k+1\), with weights \(1-s\) and \(s\), respectively. See Figure~\ref{fig:compression-point-mass}.

Since the first coordinate is unchanged and the second coordinate is moved by at most one unit, the measure \(\nu\) is locally finite.

Fix \(m\in\mathbb Z\). Since $\nu=\sum_{k\in\mathbb Z}\nu_k\times\delta_k$, we have
	\[
	\begin{aligned}
	(\mathbbm{1}_{\widetilde{\Omega}}*\nu)(x,m)
	&=
	\int_{\mathbb R\times\mathbb Z}
	\mathbbm{1}_{\widetilde{\Omega}}(x-u,m-k)\,d\nu(u,k)\\
	&=
	\sum_{k\in\mathbb Z}
	\int_{\mathbb R}
	\mathbbm{1}_{\widetilde{\Omega}}(x-u,m-k)\,d\nu_k(u).
	\end{aligned}
	\]

By the definition of \(\nu_k\) and Tonelli's theorem, we may move the sums inside the integrals:
	\[
	\begin{aligned}
	(\mathbbm{1}_{\widetilde{\Omega}}*\nu)(x,m)
	={}&
	\sum_{k\in\mathbb Z}
	\int_{\mathbb R\times[k,k+1)}
	(k+1-v)\,
	\mathbbm{1}_{\widetilde{\Omega}}(x-u,m-k)
	\,d\mu(u,v)\\
	&+
	\sum_{k\in\mathbb Z}
	\int_{\mathbb R\times[k-1,k)}
	(v-k+1)\,
	\mathbbm{1}_{\widetilde{\Omega}}(x-u,m-k)
	\,d\mu(u,v)\\
	={}&
	\int_{\mathbb R^2}
	\sum_{k\in\mathbb Z}
	\mathbbm{1}_{[k,k+1)}(v)
	(k+1-v)
	\mathbbm{1}_{\widetilde{\Omega}}(x-u,m-k)
	\,d\mu(u,v)\\
	&+
	\int_{\mathbb R^2}
	\sum_{k\in\mathbb Z}
	\mathbbm{1}_{[k-1,k)}(v)
	(v-k+1)
	\mathbbm{1}_{\widetilde{\Omega}}(x-u,m-k)
	\,d\mu(u,v).
	\end{aligned}
	\]
For each \(v\in\mathbb R\), the first sum contains only the term \(k=\lfloor v\rfloor\), while the second contains only the term \(k=\lfloor v\rfloor+1\). Moreover,
	\[
	\lfloor v\rfloor+1-v=1-\{v\},
	\qquad
	v-(\lfloor v\rfloor+1)+1=\{v\}.
	\]
	Therefore,
	\[
	\begin{aligned}
	(\mathbbm{1}_{\widetilde{\Omega}}*\nu)(x,m)
	=
	\int_{\mathbb R^2}
	\Bigl[
	&(1-\{v\})\,
	\mathbbm{1}_{\widetilde{\Omega}}
	\bigl(x-u,m-\lfloor v\rfloor\bigr)\\
	&+
	\{v\}\,
	\mathbbm{1}_{\widetilde{\Omega}}
	\bigl(x-u,m-\lfloor v\rfloor-1\bigr)
	\Bigr]
	\,d\mu(u,v).
	\end{aligned}
	\]

On the other hand,
	\[
	\begin{aligned}
	\int_0^1
	\mathbbm{1}_{\Omega}(x-u,m+t-v)\,dt
	={}&
	(1-\{v\})\mathbbm{1}_{\widetilde{\Omega}}
	(x-u,m-\lfloor v\rfloor)\\
	&+
	\{v\}\mathbbm{1}_{\widetilde{\Omega}}
	(x-u,m-\lfloor v\rfloor-1).
	\end{aligned}
	\]
    Indeed,
    \[
    	m+t-v
    	\in
    	\begin{cases}
    		\bigl(m-\lfloor v\rfloor,\,
    		m-\lfloor v\rfloor+1\bigr),
    		& t>\{v\},\\[2mm]
    		\bigl(m-\lfloor v\rfloor-1,\,
    		m-\lfloor v\rfloor\bigr),
    		& t<\{v\}.
    	\end{cases}
    \]
    The single value \(t=\{v\}\) is irrelevant to the integral.

Therefore, by Tonelli's theorem,
\[
\begin{aligned}
(\mathbbm{1}_{\widetilde{\Omega}}*\nu)(x,m)
&=
\int_{\mathbb R^2}
\int_0^1
\mathbbm{1}_{\Omega}(x-u,m+t-v)
\,dt\,d\mu(u,v)\\
&=
\int_0^1
(\mathbbm{1}_{\Omega}*\mu)(x,m+t)\,dt.
\end{aligned}
\]
Since $\mathbbm{1}_{\Omega}*\mu=1\ \text{a.e. on }\mathbb R^2$, Fubini's theorem implies that $(\mathbbm{1}_{\widetilde{\Omega}}*\nu)(x,m)=1$ for almost every \(x\in\mathbb R\). Since \(m\in\mathbb Z\) was arbitrary,
	\[
	\mathbbm{1}_{\widetilde{\Omega}}*\nu=1
	\qquad\text{a.e. on }\mathbb R\times\mathbb Z.
	\]

It remains to verify that \(\nu\) has mass \(1\) at the origin. The mass of \(\mu\) at \((0,0)\) is sent entirely to \((0,0)\), and hence $\nu(\{(0,0)\})\geq\mu(\{(0,0)\})=1$. On the other hand, 
	$$1=\mathbbm{1}_{\widetilde{\Omega}}*\nu\ge \mathbbm{1}_{\widetilde{\Omega}}*\nu|_{\{(0,0)\}}=\nu(\{(0,0)\})\mathbbm{1}_{\widetilde{\Omega}}=\nu(\{(0,0)\})\quad \text{a.e. on } \widetilde{\Omega}.$$
	Then $\nu(\{(0,0)\})=1$. Therefore \(\nu\) is a weak tiling measure for \(\widetilde{\Omega}\).
\end{proof}

\subsection{Spectrality of layered sets in \(\mathbb R\times\mathbb Z\) and \(\mathbb R^2\)}The reduction of spectrality is obtained by means of a product construction together with an appropriate linear change of variables.

\begin{proposition}\label{prop:spectral-layered-reduction}
Let \(S\subset\mathbb Z\) be finite and let
	\[
	\Omega=\bigcup_{j\in S}\Omega_j\times(j,j+1)\subset\mathbb R^2,
	\qquad
	\widetilde{\Omega}=\bigcup_{j\in S}\Omega_j\times\{j\}\subset\mathbb R\times\mathbb Z,
	\]
	where each \(\Omega_j\subset\mathbb R\) is bounded and measurable. Then \(\Omega\) is a spectral set in \(\mathbb R^2\) if and only if \(\widetilde{\Omega}\) is a spectral set in \(\mathbb R\times\mathbb Z\).
\end{proposition}

We shall use the following product result for spectral measures to prove Proposition~\ref{prop:spectral-layered-reduction}.

\begin{proposition}\label{cor:product-spectral-measure}
{\cite[Corollary~1.4]{KLLL2026}}
Let \(\nu\) be a finite Borel measure on \(\mathbb R^m\) with compact support. Then \(\mathfrak m|_{[0,1]^d}\times \nu\) is a spectral measure on \(\mathbb R^{d+m}\) if and only if \(\nu\) is a spectral measure on \(\mathbb R^m\).
\end{proposition}

\begin{proof}[Proof of Proposition~\ref{prop:spectral-layered-reduction}]
Suppose that \(\widetilde{\Omega}\) is a spectral set in \(\R\times\Z\) with spectrum \(\Sigma\subset\R\times\T\). One can check by definition that the set \(\Omega\) is a spectral set in \(\R^2\) with spectrum
	\[
	\Lambda:=\{(\lambda,\eta+k):(\lambda,\eta)\in\Sigma,\ k\in\Z\}.
	\]

Conversely, suppose that \(\Omega\) is a spectral set in \(\R^2\). We work in \(\mathbb R^3\), with the \(\Omega_j\subset \mathbb R\times\{0\}^2\) and \(\Omega\subset \mathbb R^2\times\{0\}\). Define the ``tilted'' line segment \(I\) which connects \((0,0,0)\) to \((0,1,1)\). Then let 
$$
E=\bigcup_{j\in S}\Omega_j+(0,j,0)+I.
$$
Let  \(\mu\) be the two-dimensional Lebesgue measure on \(\Omega\), and let \(\nu:=\mathcal H^2|_E\) be the Hausdorff measure on $E$. See Figures~\ref{fig:Omega-plane-measure} and~\ref{fig:E-above-Omega}.
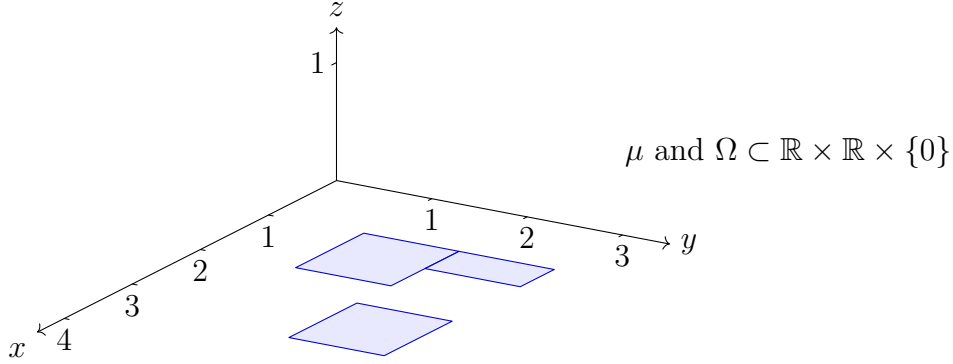
\begin{figure}[H]
	\centering
	\begin{tikzpicture}[
		scale=1.2,
		x={(-0.75cm,-0.38cm)},
		y={(1.05cm,-0.20cm)},
		z={(0cm,1.3cm)},
		line join=round,
		line cap=round
	]

	\draw[->] (0,0,0) -- (4.4,0,0) node[below left] {$x$};
	\draw[->] (0,0,0) -- (0,3.5,0) node[right] {$y$};
	\draw[->] (0,0,0) -- (0,0,1.3) node[above] {$z$};

	\foreach \x in {1,2,3,4}{
		\draw (\x,0,0) -- ++(0,0.05,0);
		\node[below] at (\x,0,0) {$\x$};
	}

	\foreach \y in {1,2,3}{
		\draw (0,\y,0) -- ++(0.05,0,0);
		\node[below] at (0,\y,0) {$\y$};
	}

	\draw (0,0,1) -- ++(0.07,0,0);
	\node[left] at (0,0,1) {$1$};

	\filldraw[
		fill=blue!20,
		draw=blue!80!black,
		fill opacity=0.45,
		draw opacity=1
	] (1,1,0) -- (2,1,0) -- (2,2,0) -- (1,2,0) -- cycle;

	\filldraw[
		fill=blue!20,
		draw=blue!80!black,
		fill opacity=0.45,
		draw opacity=1
	] (1,2,0) -- (1.5,2,0) -- (1.5,3,0) -- (1,3,0) -- cycle;

	\filldraw[
		fill=blue!20,
		draw=blue!80!black,
		fill opacity=0.45,
		draw opacity=1
	] (2.5,2,0) -- (3.5,2,0) -- (3.5,3,0) -- (2.5,3,0) -- cycle;

	\node[right] at (0.4,3.2,0.85)
	{\(\mu\) and \(\Omega\subset\mathbb R\times\mathbb R\times\{0\}\)};

	\end{tikzpicture}
	\caption{The set \(\Omega\) with \(S=\{1,2\}\), \(\Omega_1=[1,2]\), and
	\(\Omega_2=[1,1.5]\cup[2.5,3.5]\). The measure \(\mu\) has density $\mathbbm{1}_\Om$ with respect to planar Lebesgue measure.}
	\label{fig:Omega-plane-measure}
\end{figure}
The orthogonal projection \(\pi:\mathbb R^3\to \mathbb R^2\times\{0\}\)
is injective on \(E\) outside a \(\nu\)-null set, and its pushforward satisfies $\pi_*\nu=\sqrt{2}\,\mu$. Therefore, if \(\Lambda\subset \mathbb R^2\times\{0\}\) is a spectrum of \(\mu\), then \(\Lambda\) is also a spectrum of \(\nu\), since the two \(L^2\) spaces can be identified via \(\pi\), and the exponentials in \(\Lambda\) take the same value at \(x\) and \(\pi(x)\).
\begin{figure}[H]
	\centering
	\begin{tikzpicture}[
		scale=1.2,
		x={(-0.75cm,-0.38cm)},
		y={(1.05cm,-0.20cm)},
		z={(0cm,1.3cm)},
		line join=round,
		line cap=round
	]

		\draw[->] (0,0,0) -- (4.4,0,0) node[below left] {$x$};
		\draw[->] (0,0,0) -- (0,3.6,0) node[right] {$y$};
		\draw[->] (0,0,0) -- (0,0,1.3) node[above] {$z$};

		\foreach \x in {1,2,3,4}{
			\draw (\x,0,0) -- ++(0,0.05,0);
			\node[below] at (\x,0,0) {$\x$};
		}

		\foreach \y in {1,2,3}{
			\draw (0,\y,0) -- ++(0.05,0,0);
			\node[below] at (0,\y,0) {$\y$};
		}

		\draw (0,0,1) -- ++(0.07,0,0);
		\node[left] at (0,0,1) {$1$};


		\filldraw[
			fill=blue!20,
			draw=blue!80!black,
			fill opacity=0.45,
			draw opacity=1
		] (1,1,0) -- (2,1,0) -- (2,2,0) -- (1,2,0) -- cycle;

		\filldraw[
			fill=blue!20,
			draw=blue!80!black,
			fill opacity=0.45,
			draw opacity=1
		] (1,2,0) -- (1.5,2,0) -- (1.5,3,0) -- (1,3,0) -- cycle;

		\filldraw[
			fill=blue!20,
			draw=blue!80!black,
			fill opacity=0.45,
			draw opacity=1
		] (2.5,2,0) -- (3.5,2,0) -- (3.5,3,0) -- (2.5,3,0) -- cycle;


		\filldraw[
			fill=orange!60,
			draw=orange!70!black,
			fill opacity=0.62,
			draw opacity=1
		] (1,1,0) -- (2,1,0) -- (2,2,1) -- (1,2,1) -- cycle;

		\filldraw[
			fill=orange!60,
			draw=orange!70!black,
			fill opacity=0.62,
			draw opacity=1
		] (1,2,0) -- (1.5,2,0) -- (1.5,3,1) -- (1,3,1) -- cycle;

		\filldraw[
			fill=orange!60,
			draw=orange!70!black,
			fill opacity=0.62,
			draw opacity=1
		] (2.5,2,0) -- (3.5,2,0) -- (3.5,3,1) -- (2.5,3,1) -- cycle;

		\node[right] at (0.85,3.15,0.9) {\(E\) together with \(\Omega\)};

	\end{tikzpicture}
	\caption{The set \(E=\bigcup_{j\in S}\Omega_j+(0,j,0)+I\) lying above \(\Omega\).}
	\label{fig:E-above-Omega}
\end{figure}
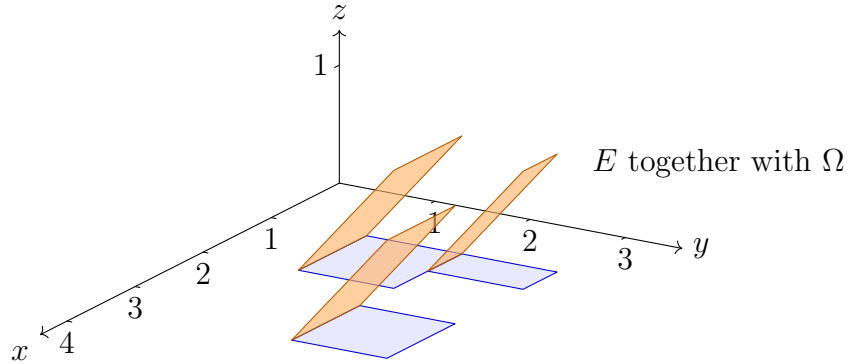

Let \(J=\{0\}^2\times[0,1]\) be the vertical unit line segment at the origin, and define the linear transformation \(T\) which maps \(e_1\) to \(e_1\), \(e_2\) to \(e_2\), and \(e_2+e_3\) to \(e_3\), i.e.
	\[
	T=
	\begin{bmatrix}
	1&0&0\\
	0&1&-1\\
	0&0&1
	\end{bmatrix}
	\quad\text{with}\quad
	T^{-1}=
	\begin{bmatrix}
	1&0&0\\
	0&1&1\\
	0&0&1
	\end{bmatrix}.
	\]
	It follows that \(TI=J\) and $T_*\nu=\sqrt{2}\,\sigma$, where \(\sigma\) is the Hausdorff measure on
    \[
    \left(\bigcup_{j\in S}\Omega_j+(0,j,0)\right)\times[0,1].
    \]
	See Figure~\ref{fig:product-measure-sigma}. It follows that \(L=T^{-\mathsf T}\Lambda\) is a spectrum for the product
	measure	$\sigma=\omega\times h_J$, where \(\omega\) is the Hausdorff measure on \(\bigcup_{j\in S}\Omega_j+(0,j,0)\), and
	\(h_J\) is the Hausdorff measure on \(J\).
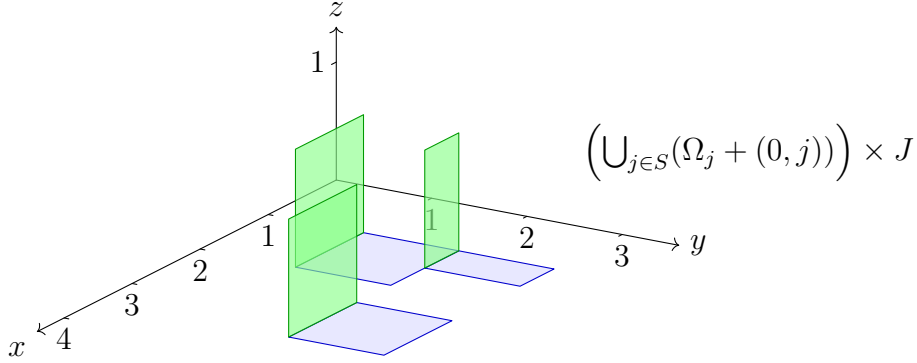
\begin{figure}[H]
	\centering
	\begin{tikzpicture}[
		scale=1.2,
		x={(-0.75cm,-0.38cm)},
		y={(1.05cm,-0.20cm)},
		z={(0cm,1.3cm)},
		line join=round,
		line cap=round
	]

		\draw[->] (0,0,0) -- (4.4,0,0) node[below left] {$x$};
		\draw[->] (0,0,0) -- (0,3.6,0) node[right] {$y$};
		\draw[->] (0,0,0) -- (0,0,1.3) node[above] {$z$};

		\foreach \x in {1,2,3,4}{
			\draw (\x,0,0) -- ++(0,0.05,0);
			\node[below] at (\x,0,0) {$\x$};
		}

		\foreach \y in {1,2,3}{
			\draw (0,\y,0) -- ++(0.05,0,0);
			\node[below] at (0,\y,0) {$\y$};
		}

		\draw (0,0,1) -- ++(0.07,0,0);
		\node[left] at (0,0,1) {$1$};


		\filldraw[
			fill=blue!20,
			draw=blue!80!black,
			fill opacity=0.45,
			draw opacity=1
		] (1,1,0) -- (2,1,0) -- (2,2,0) -- (1,2,0) -- cycle;

		\filldraw[
			fill=blue!20,
			draw=blue!80!black,
			fill opacity=0.45,
			draw opacity=1
		] (1,2,0) -- (1.5,2,0) -- (1.5,3,0) -- (1,3,0) -- cycle;

		\filldraw[
			fill=blue!20,
			draw=blue!80!black,
			fill opacity=0.45,
			draw opacity=1
		] (2.5,2,0) -- (3.5,2,0) -- (3.5,3,0) -- (2.5,3,0) -- cycle;


		\filldraw[
			fill=green!45,
			draw=green!60!black,
			fill opacity=0.65,
			draw opacity=1
		] (1,1,0) -- (2,1,0) -- (2,1,1) -- (1,1,1) -- cycle;

		\filldraw[
			fill=green!45,
			draw=green!60!black,
			fill opacity=0.65,
			draw opacity=1
		] (1,2,0) -- (1.5,2,0) -- (1.5,2,1) -- (1,2,1) -- cycle;

		\filldraw[
			fill=green!45,
			draw=green!60!black,
			fill opacity=0.65,
			draw opacity=1
		] (2.5,2,0) -- (3.5,2,0) -- (3.5,2,1) -- (2.5,2,1) -- cycle;

		\node[right] at (0.95,3.15,1)
		{\(\left(\bigcup_{j\in S}(\Omega_j+(0,j))\right)\times J\)};

	\end{tikzpicture}
	\caption{The product measure \(\sigma\) in green.}
	\label{fig:product-measure-sigma}
\end{figure}
Now by Proposition~\ref{cor:product-spectral-measure}, it follows from the spectrality of \(\sigma\) that \(\omega\) is spectral with a spectrum \(M\subseteq \mathbb R^2\times\{0\}\). But if a measure \(\omega\) is spectral, then we can always find a spectrum in the dual group of the group generated by
\(\operatorname{supp}\omega\), which is contained in \(\mathbb R\times\mathbb Z\times\{0\}\).
Therefore, we can find a spectrum in \(\mathbb R\times\mathbb T\times\{0\}\), which implies that \(\widetilde{\Omega}\) is spectral in \(\mathbb R\times\mathbb Z\).
\end{proof}

\section{Structure of (weak) tiles  and spectral sets in $\R\times \Z$ }\label{S7}

Let $\widetilde{\Omega} \subset \R\times\Z$ be the union of three pairwise non-overlapping unit intervals. Since the weak tiling property is invariant under translations, we may, without loss of generality, assume that
\begin{equation}\label{eq:three-intervals-Omega1}
\widetilde{\Om} = (0,1)\times\{0\} \cup (a,a+1)\times\{\ell_1\} \cup (b,b+1)\times\{\ell_2\},
\quad a,b\in\R,\ \ell_1,\ell_2\in\Z.
\end{equation}
The main result of this section is the following theorem, which gives a complete characterization of when $\widetilde{\Om}$ is a weak tile of $\R\times\Z$. 
\begin{theorem}\label{prop-WT-3intervals-characterization}
Let $\widetilde{\Om}$ be as in \eqref{eq:three-intervals-Omega1}. Then $\widetilde{\Om}$ is a weak tile of $\R\times\Z$ if and only if exactly one of the following holds:
	\begin{itemize}
		\item[\upshape (1)] $0=\ell_1=\ell_2$, and $\{0,a,b\}$ tiles $\Z$;
		
		\item[\upshape (2)] Exactly two of $0,\ell_1,\ell_2$ are equal, and one of the following holds:
		\begin{itemize}
			\item[\upshape (i)] $0\neq \ell_1=\ell_2$ and $a-b\in\Z^*$;
			\item[\upshape (ii)] $\ell_1=0\neq \ell_2$ and $a\in\Z^*$;
			\item[\upshape (iii)] $\ell_2=0\neq \ell_1$ and $b\in\Z^*$.
		\end{itemize}
		
		\item[\upshape (3)] $0,\ell_1,\ell_2$ are pairwise distinct, and one of the following holds:
		\begin{itemize}
			\item[\upshape (i)] $\{0,\ell_1,\ell_2\}$ tiles $\Z$;
			\item[\upshape (ii)] $\{0,\ell_1,\ell_2\}$ does not tile $\Z$, and there exist $\xi\in\R$ and $M,N\in\Z$ such that
			\[
			a=M+\ell_1\xi,\qquad b=N+\ell_2\xi,
			\]
			where $(M,\ell_1)$ and $(N,\ell_2)$ are $\Q$-linearly independent.
		\end{itemize}
	\end{itemize}
Moreover, in each case, \(\widetilde{\Omega}\) is both a tile and a spectral set of \(\mathbb R\times\mathbb Z\).
\end{theorem}


\subsection{Outline and preliminaries of the proof.} We will prove Theorem~\ref{prop-WT-3intervals-characterization} by considering the possible configurations of the set $\{0,\ell_1,\ell_2\}$ in the following subsections.
\begin{enumerate}
    \item Case (1) will be studied in Lemma \ref{lem:three-intervals-l1=l2=0} which follows essentially from Theorem \ref{prop-three-intervals-R-tile-Z}.
    \item Case (2)(i) will be given in Lemma \ref{lem:three-intervals-l1=l2-not0}. The remaining cases can be reduced to this case through a translation.
    \item Case (3)(i) will be proved in Lemma \ref{lem:three-intervals-0-l1-l2-tile-Z}, while the most subtle Case (3)(ii) will be given in Subsection \ref{subsection7.4}.  
\end{enumerate}

For the bounded sets in $\R\times\Z$ considered here, both tiling and spectrality imply weak tiling. The former is immediate, while the latter follows from Theorems~\ref{thm:1.1} and~\ref{thm:layered-reduction}. Thus, in each case, it suffices to show that every weak tile has the stated form and that every set of this form is both a tile and a spectral set.

We next establish some basic properties of bounded measurable sets in $\R\times\Z$. Let $S\subset\mathbb{Z}$ be a finite set, and define the set
\begin{equation}\label{eq:Omega}
\widetilde{\Omega}=\bigcup_{j\in S}\Omega_j\times\{j\},
\end{equation}
where each $\Omega_j\subset\mathbb{R}$ is bounded and measurable. For any $\xi\in\mathbb{R}$, we introduce the layer-translated set
\[
\widetilde{\Omega}_\xi:=\begin{bmatrix} 1&\xi\\0&1\end{bmatrix}\widetilde{\Omega} =\bigcup_{j\in S}(\Omega_j+j\xi)\times\{j\}.
\]
The shear map $T_\xi(x,j):=(x+j\xi,j)$ is a Haar-measure-preserving group automorphism of $\R\times\Z$ and therefore preserves tiling, weak tiling, and spectrality. We thus obtain the following proposition.

\begin{proposition}\label{prop:translation-invariance}
For every $\xi\in\mathbb{R}$, the following statements hold.
\begin{enumerate}
    \item[\upshape (i)] $\widetilde{\Omega}$ is a (weak) tile of $\mathbb{R}\times\mathbb{Z}$ if and only if $\widetilde{\Omega}_\xi$ is a (weak) tile of $\mathbb{R}\times\mathbb{Z}$;
    \item[\upshape (ii)] $\widetilde{\Omega}$ is a spectral set of $\mathbb{R}\times\mathbb{Z}$ if and only if $\widetilde{\Omega}_\xi$ is a spectral set of $\mathbb{R}\times\mathbb{Z}$.
\end{enumerate}
\end{proposition}

The use of projections below is inspired by Rao and Xue~\cite[Theorem~1.3 and Proposition~3.1]{RX2006}. They characterized prime-cardinality subsets of $\Z^2$ admitting periodic tilings by the existence of an integer linear projection that is injective on the set and whose image tiles $\Z$. Here, we use suitable projections to obtain tiling and spectrality in $\R\times\Z$ from the corresponding one-dimensional properties.

For each $\eta\in\mathbb{R}$, define the projection
	\[
	\pi_\eta:\mathbb{R}\times\mathbb{Z}\to\mathbb{R},\qquad
	\pi_\eta(x,j)=x+\eta j.
	\]
	Then we have
	\[
	\pi_\eta(\widetilde{\Omega})=\bigcup_{j\in S}(\Omega_j+j\eta).
	\]

Suppose that the sets $\Omega_j+j\eta$, $j\in S$, are pairwise non-overlapping. If $\pi_\eta(\widetilde{\Omega})$ tiles $\mathbb R$ with a tiling set $T$, then $T\times\mathbb Z$ is a tiling set for $\widetilde{\Omega}_\eta$ in $\mathbb R\times\mathbb Z$. The same product construction works for weak tilings. For spectrality, the pairwise non-overlapping assumption on the projected components allows a spectrum of $\pi_\eta(\widetilde{\Omega})$ to be lifted to a spectrum of $\widetilde{\Omega}_\eta$. More precisely, if $\Lambda$ is a spectrum of $\pi_\eta(\widetilde{\Omega})$, then $\Lambda\times\{0\}\subset\mathbb R\times\mathbb T$ is a spectrum of $\widetilde{\Omega}_\eta$. Combining these observations with Proposition~\ref{prop:translation-invariance}, we obtain the following sufficient conditions.


\begin{corollary}\label{cor:lift-1D-to-2D}
Let $\widetilde{\Omega}\subset\mathbb{R}\times\mathbb{Z}$ be a measurable set of the form \eqref{eq:Omega}.
\begin{enumerate}
    \item[\upshape (1)] If there exists $\eta\in\mathbb{R}$ such that the sets $\Omega_j + j\eta$ ($j\in S$) are pairwise non-overlapping, and the set $\pi_\eta(\widetilde{\Omega})$ is a (weak) tile of $\mathbb{R}$, then $\widetilde{\Omega}$ is a (weak) tile of $\mathbb{R}\times\mathbb{Z}$.
    \item[\upshape (2)] If there exists $\eta\in\mathbb{R}$ such that the sets $\Omega_j + j\eta$ ($j\in S$) are pairwise non-overlapping, and the set $\pi_\eta(\widetilde{\Omega})$ is a spectral set of $\mathbb{R}$, then $\widetilde{\Omega}$ is a spectral set of $\mathbb{R}\times\mathbb{Z}$.
\end{enumerate}
\end{corollary}

We now further restrict our attention to  $\widetilde{\Omega}$ being a union of finitely many pairwise non-overlapping unit intervals in $\R\times\Z$, i.e.
	\begin{equation}\label{eq:finite-intervals-RZ}
		\widetilde{\Omega}=(0,1)\times\{0\}+A,\quad A\subset\R\times\Z\ \text{is a finite set}.
	\end{equation}
For a Borel measure $\mu$ on $\R\times\Z$ and each $j\in\Z$, we define the fiber measure $\mu_j$ on $\R$ by
	\[
	\mu_j(E) \;:=\; \mu\bigl(E\times\{j\}\bigr),\qquad E\subset\R\ \text{Borel}.
	\]
	We also define the measure $\nu_j$ on $\R$ by
	\[
	\nu_j \;:=\sum_{(a,\ell)\in A}\mu_{j-\ell}*\delta_a.
	\]

\begin{proposition}\label{prop:WT-finite-intervals-slicing}
Let $\widetilde{\Omega}$ be as in \eqref{eq:finite-intervals-RZ} and $\mu$ be a {non-negative}, locally finite Borel measure on $\R\times\Z$. Suppose that $\mathbbm{1}_{\widetilde{\Omega}}*\mu=1$ a.e. on $\R\times\Z$ and {$\mu(\{(0,0)\})=1$}. Then for each $j\in\Z$, the measure $\nu_j$ is a {non-negative}, locally finite Borel measure on $\R$ such that 
	$$\mathbbm{1}_{(0,1)} * \nu_j = 1 \quad \text{a.e. on }\R.$$
	Moreover, if $(a,\ell)$ and $(a',\ell)$ are two distinct elements of $A$, then $a-a'\in\mathbb{Z}^*$.
\end{proposition}
\begin{proof}
Fix $j\in\Z$. Since $A$ is finite and each $\mu_{j-\ell}$ is a {non-negative}, locally finite Borel measure on $\R$, it follows that $\nu_j$ is also a {non-negative}, locally finite Borel measure on $\R$. 

Note that
	\[
	\mathbbm{1}_{\widetilde{\Omega}}
	=\sum_{(a,\ell)\in A}\mathbbm{1}_{(0,1)\times\{0\}} * \delta_{(a,\ell)}.
	\]
	Hence, for a.e. $x\in\R$,
	\begin{align*}
	1
	&= (\mathbbm{1}_{\widetilde{\Omega}} * \mu)(x,j)\\
	&= \sum_{(a,\ell)\in A}
	\bigl(\mathbbm{1}_{(0,1)\times\{0\}} * \delta_{(a,\ell)} * \mu\bigr)(x,j)\\
	&= \sum_{(a,\ell)\in A}
	\int_{\R\times\Z}
	\mathbbm{1}_{(0,1)\times\{0\}}(x-y-a,j-k-\ell)\,d\mu(y,k).
	\end{align*}
	Since
	\[
	\mathbbm{1}_{(0,1)\times\{0\}}(x-y-a,j-k-\ell)
	=\mathbbm{1}_{(0,1)}(x-y-a)\,\mathbbm{1}_{\{0\}}(j-k-\ell),
	\]
	only the slice $k=j-\ell$ contributes, and therefore, for a.e. $x\in\R$,
	\begin{align*}
	1
	&= \sum_{(a,\ell)\in A}
	\int_{\R}\mathbbm{1}_{(0,1)}(x-y-a)\,d\mu_{j-\ell}(y)\\
	&= \sum_{(a,\ell)\in A}
	\bigl(\mathbbm{1}_{(0,1)} * \mu_{j-\ell} * \delta_a\bigr)
    (x)\\
	&= \mathbbm{1}_{(0,1)} *
	\Big(\sum_{(a,\ell)\in A}\mu_{j-\ell} * \delta_a\Big)(x)\\
	&= \mathbbm{1}_{(0,1)} * \nu_j(x).
	\end{align*}

Now suppose  $(a,\ell),(a',\ell)$ are two distinct elements of $A$. Since $\mu(\{(0,0)\})=1$, by the definition of $\nu_\ell$, we have $\nu_\ell(\{a\}),\nu_\ell(\{a'\})\geq1$, while $\mathbbm{1}_{(0,1)}*\nu_\ell=1$ a.e. implies that every atom of $\nu_\ell$ has mass at most $1$; hence $\nu_\ell(\{a\})=\nu_\ell(\{a'\})=1$. By Theorem~\ref{th-Keller}, it follows that $a-a'\in\Z^*$.
\end{proof}

\subsection{Structure of weak tiles for three unit intervals in \texorpdfstring{$\R\times\Z$}{RtimesZ}} We are now fully equipped to prove Theorem \ref{prop-WT-3intervals-characterization}.



\subsubsection{\texorpdfstring{Case (1): $\ell_1=\ell_2=0$}{l1=l2=0}}

In this case, $\widetilde{\Omega}$ reduces to a union of three unit intervals in $\R$. Hence $\widetilde{\Omega}$ weakly tiles $\R\times\Z$ if and only if it weakly tiles $\R$ by translations, which is equivalent to saying that $\{0,a,b\}$ tiles $\Z$ by Proposition~\ref{prop:WT-finite-intervals-slicing} and Theorem~\ref{prop-three-intervals-R-tile-Z}. 


\begin{lemma}\label{lem:three-intervals-l1=l2=0}
Suppose that $\widetilde{\Omega}$ is as in \eqref{eq:three-intervals-Omega1} and that $\ell_1=\ell_2=0$. Then $\widetilde{\Omega}$ is a weak tile of $\R\times\Z$ if and only if $\{0,a,b\}$ is a tile of $\Z$. In this case, $\widetilde{\Omega}$ is both a tile and a spectral set in $\R\times\Z$.
\end{lemma}


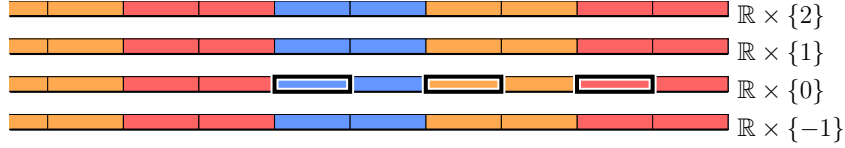
\begin{figure}[H]
	\centering
	\begin{tikzpicture}[scale=1]
		\pgfmathsetmacro{\a}{2}
		\pgfmathsetmacro{\b}{4}

		\foreach \m in {-1,0,1,2}{
			\begin{scope}
				\clip (-3.5,\m/2) rectangle (6,\m/2+0.2);
				\begin{scope}[shift={(0,\m/2)}]
					\foreach \n in {-2,-1,0,1,2}{
						\foreach \s in {0,1}{
							\fill[myblue] ({6*\n+\s},0) rectangle ({6*\n+\s+1},0.2);
							\fill[myorange] ({6*\n+2+\s},0) rectangle ({6*\n+3+\s},0.2);
							\fill[myred] ({6*\n+4+\s},0) rectangle ({6*\n+5+\s},0.2);

							\draw[black] ({6*\n+\s},0) rectangle ({6*\n+\s+1},0.2);
							\draw[black] ({6*\n+2+\s},0) rectangle ({6*\n+3+\s},0.2);
							\draw[black] ({6*\n+4+\s},0) rectangle ({6*\n+5+\s},0.2);
						}
					}
				\end{scope}
			\end{scope}

			\draw[thick] (-3.5,\m/2)--(6,\m/2)
			node[right,scale=0.8] {$\mathbb R\times\{\m\}$};
		}

		\draw[white,line width=3pt] (0,0) rectangle (1,0.2);
		\draw[white,line width=3pt] (\a,0) rectangle ({\a+1},0.2);
		\draw[white,line width=3pt] (\b,0) rectangle ({\b+1},0.2);

		\draw[black,line width=1.5pt] (0,0) rectangle (1,0.2);
		\draw[black,line width=1.5pt] (\a,0) rectangle ({\a+1},0.2);
		\draw[black,line width=1.5pt] (\b,0) rectangle ({\b+1},0.2);
	\end{tikzpicture}

	\caption{A translational tiling of \(\mathbb R\times\mathbb Z\) by three unit intervals with \(\ell_1=\ell_2=0\), illustrated for \(\{0,a,b\}=\{0,2,4\}\). The framed intervals are precisely \(\widetilde{\Omega}\).}
	\label{fig:three-unit-intervals-same-level}
\end{figure}
\begin{proof}
By the one-dimensional result of Theorem~\ref{th-three-cubes-parallel}, it suffices to show that if $\widetilde{\Omega}$ is a weak tile of $\R\times\Z$, then $\{0,a,b\}$ is a tile of $\Z$. Suppose that $\mu$ is a {non-negative}, locally finite Borel measure on $\R\times\Z$ such that 
	$$\mathbbm{1}_{\widetilde{\Omega}}*\mu=1 \qae \text{on } \R\times\Z\quad\text{with }\mu(\{(0,0)\})=1.$$
	By Proposition~\ref{prop:WT-finite-intervals-slicing}, for $j=0$, we have 
	$$\mathbbm{1}_{(0,1)}*(\mu_0+\mu_0*\delta_a+\mu_0*\delta_b)=1\qae\quad\text{with }\mu_0(\{0\})=1.$$
	It follows that 
	$$\mathbbm{1}_{(0,1)+\{0,a,b\}}*\mu_0=1\qae\quad\text{with }\mu_0(\{0\})=1,$$
	which implies $(0,1)+\{0,a,b\}$ is a weak tile of $\R$. By Proposition~\ref{prop:WT-finite-intervals-slicing}, we have $a,b\in\Z$. Hence, by Theorem~\ref{prop-three-intervals-R-tile-Z}, $\{0,a,b\}$ is a tile of $\Z$.
\end{proof}

\subsubsection{Case (2): Exactly two of \texorpdfstring{$0,\ell_1,\ell_2$}{0,l1,l2} coincide} We only discuss the case $0\neq \ell_1=\ell_2$. The other two cases can be reduced to this case by translation.
\begin{lemma}\label{lem:three-intervals-l1=l2-not0}
	Suppose that $\widetilde{\Omega}$ is as in \eqref{eq:three-intervals-Omega1} and $0\neq \ell_1=\ell_2$. Then the following are equivalent:
	\begin{enumerate}
		\item[\upshape (1)] $\widetilde{\Omega}$ is a weak tile of $\R\times\Z$;
		\item[\upshape (2)] $a-b\in\Z^*$;
		\item[\upshape (3)] $\widetilde{\Omega}$ is a tile of $\R\times\Z$;
		\item[\upshape (4)] $\widetilde{\Omega}$ is a spectral set of $\R\times\Z$.
	\end{enumerate}
\end{lemma}






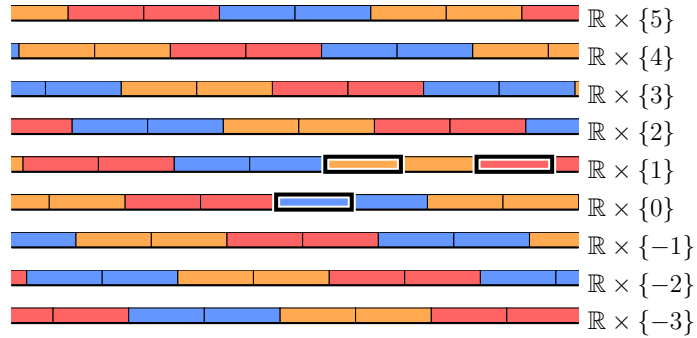
\begin{figure}[H]
	\centering
	\begin{tikzpicture}[scale=1]
		\pgfmathsetmacro{\a}{0.65}
		\pgfmathsetmacro{\b}{2.65}

		\foreach \m in {-3,-2,-1,0,1,2,3,4,5}{
			\pgfmathtruncatemacro{\rzero}{mod(-\m+30,3)}
			\pgfmathtruncatemacro{\rone}{mod(\rzero+1,3)}
			\pgfmathtruncatemacro{\rtwo}{mod(\rzero+2,3)}

			\begin{scope}
				\clip (-3.5,\m/2) rectangle (4,\m/2+0.2);
				\begin{scope}[shift={({\m*\a},\m/2)}]
					\foreach \n in {-3,-2,-1,0,1,2,3}{
						\foreach \s in {0,1}{
							\fill[myblue] ({6*\n+2*\rzero+\s},0) rectangle ({6*\n+2*\rzero+\s+1},0.2);
							\fill[myorange] ({6*\n+2*\rone+\s},0) rectangle ({6*\n+2*\rone+\s+1},0.2);
							\fill[myred] ({6*\n+2*\rtwo+\s},0) rectangle ({6*\n+2*\rtwo+\s+1},0.2);

							\draw[black] ({6*\n+2*\rzero+\s},0) rectangle ({6*\n+2*\rzero+\s+1},0.2);
							\draw[black] ({6*\n+2*\rone+\s},0) rectangle ({6*\n+2*\rone+\s+1},0.2);
							\draw[black] ({6*\n+2*\rtwo+\s},0) rectangle ({6*\n+2*\rtwo+\s+1},0.2);
						}
					}
				\end{scope}
			\end{scope}

			\draw[thick] (-3.5,\m/2)--(4,\m/2)
			node[right,scale=0.8] {$\mathbb R\times\{\m\}$};
		}


		\draw[white,line width=3pt] (0,0) rectangle (1,0.2);
		\draw[white,line width=3pt] (\a,0.5) rectangle ({\a+1},0.7);
		\draw[white,line width=3pt] (\b,0.5) rectangle ({\b+1},0.7);

		\draw[black,line width=1.5pt] (0,0) rectangle (1,0.2);
		\draw[black,line width=1.5pt] (\a,0.5) rectangle ({\a+1},0.7);
		\draw[black,line width=1.5pt] (\b,0.5) rectangle ({\b+1},0.7);
	\end{tikzpicture}

	\caption{A translational tiling of \(\mathbb R\times\mathbb Z\) by three unit intervals with $\ell_1=\ell_2=1$. The framed intervals are precisely \(\widetilde{\Omega}\).}
	\label{fig:three-unit-intervals}
\end{figure}

\begin{proof}
	Since $(3), (4)\Rightarrow (1)$, it suffices to show that $(1)\Rightarrow(2)\Rightarrow (3),(4)$.

	Since $0\neq \ell_1=\ell_2$, $(1)\Rightarrow(2)$ follows from Proposition~\ref{prop:WT-finite-intervals-slicing}. 
	Now we prove $(2)\Rightarrow(3)$. Suppose $a-b\in\Z^*$. By Corollary~\ref{cor:lift-1D-to-2D}, it suffices to find $\eta$ such that $$E_\eta:=\{0,\,a+\ell_1\eta,\,b+\ell_1\eta\}$$ is a three-element tile of $\Z$. Without loss of generality, we may assume that $a-b>0$. Write $a-b=3^s q$, where $s\geq 0$ and $3\nmid q$. Let $u$ be an integer such that 
	$$
	\begin{cases}
	u=2 & \text{if } q\equiv 1\pmod{3},\\
	u=1 & \text{if } q\equiv 2\pmod{3}.
	\end{cases}$$ 
	Then \(\{0,u,u-q\}\) is a complete residue system modulo \(3\). We take \(\eta=(3^s u-a)/\ell_1\). Then
		\[
		E_\eta=\{0,3^s u,3^s(u-q)\}
		\]
		is a tile of \(\mathbb Z\), since \(\{0,u,u-q\}\) is a complete residue system modulo \(3\).

	For $(2)\Rightarrow (4)$, since Fuglede's conjecture holds for the union of three non-overlapping unit intervals in $\R$ by Theorem \ref{th-three-cubes-parallel}, $\pi_\eta(\widetilde{\Om})$ is also a spectral set of $\R$. By Corollary~\ref{cor:lift-1D-to-2D}, $\widetilde{\Omega}$ is a spectral set of $\R\times\Z$.
\end{proof}

\subsubsection{Case (3): \texorpdfstring{$0,\ell_1,\ell_2$}{0,l1,l2} are distinct integers}
Now we focus on Case 3, i.e., $0,\ell_1,\ell_2$ are distinct integers.
If $\{0,\ell_1,\ell_2\}$ forms a tile of $\Z$, then it is immediate that $\widetilde{\Omega}$ tiles $\R\times\Z$. For later reference, we state this as a lemma.
\begin{lemma}\label{lem:three-intervals-0-l1-l2-tile-Z}
	Suppose that $\widetilde{\Omega}$ is as in \eqref{eq:three-intervals-Omega1}. If $\{0,\ell_1,\ell_2\}$ is a three-element tile of $\Z$, then $\widetilde{\Omega}$ is both a tile and a spectral set of $\R\times\Z$.
\end{lemma}
\begin{figure}[!ht]
	\centering
	\begin{tikzpicture}[scale=1]
		\pgfmathsetmacro{\a}{0.65}
		\pgfmathsetmacro{\b}{1.85}
		\pgfmathsetmacro{\ellone}{1}
		\pgfmathsetmacro{\elltwo}{2}

		\foreach \m in {-3,-2,-1,0,1,2,3,4,5}{
			\pgfmathtruncatemacro{\r}{mod(\m+300,3)}

			\ifcase\r
				\def\thiscolor{myblue}
				\pgfmathsetmacro{\xshift}{0}
			\or
				\def\thiscolor{myorange}
				\pgfmathsetmacro{\xshift}{\a}
			\or
				\def\thiscolor{myred}
				\pgfmathsetmacro{\xshift}{\b}
			\fi

			\begin{scope}
				\clip (-3.5,\m/2) rectangle (4,\m/2+0.2);
				\begin{scope}[shift={(\xshift,\m/2)}]
					\foreach \n in {-7,-6,-5,-4,-3,-2,-1,0,1,2,3,4,5}{
						\fill[\thiscolor] (\n,0) rectangle (\n+1,0.2);
						\draw[black] (\n,0) rectangle (\n+1,0.2);
					}
				\end{scope}
			\end{scope}

			\draw[thick] (-3.5,\m/2)--(4,\m/2)
			node[right,scale=0.8] {$\mathbb R\times\{\m\}$};
		}

		\draw[white,line width=3pt] (0,0) rectangle (1,0.2);
		\draw[white,line width=3pt] (\a,0.5) rectangle ({\a+1},0.7);
		\draw[white,line width=3pt] (\b,1) rectangle ({\b+1},1.2);

		\draw[black,line width=1.5pt] (0,0) rectangle (1,0.2);
		\draw[black,line width=1.5pt] (\a,0.5) rectangle ({\a+1},0.7);
		\draw[black,line width=1.5pt] (\b,1) rectangle ({\b+1},1.2);
	\end{tikzpicture}

	\caption{A translational tiling of \(\mathbb R\times\mathbb Z\) in the case where \(\{0,\ell_1,\ell_2\}\) tiles \(\mathbb Z\). Here \(\{\ell_1,\ell_2\}=\{1,2\}\), with tiling complement \(3\mathbb Z\), and the translation set is \(\mathbb Z\times3\mathbb Z\). The framed intervals are precisely \(\widetilde{\Omega}\).}
	\label{fig:three-levels-tile-Z}
\end{figure}
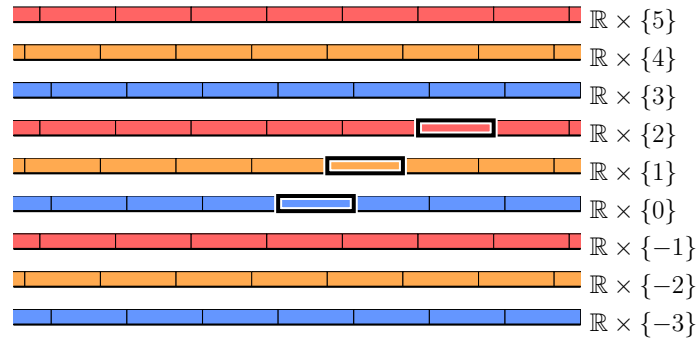

\begin{proof}
	Suppose that $\{0,\ell_1,\ell_2\}$ tiles $\Z$ with a tiling set $T\subset \Z$. Define
	\[
	T'=\Z\times T\subset \R\times\Z.
	\]
	Then it is straightforward to verify that $\widetilde{\Omega}$ tiles $\R\times\Z$ by translations with tiling set $T'$.
	
	On the other hand, since $\{0,\ell_1,\ell_2\}$ is a tile of $\Z$, it is also a spectral set of $\Z$. Let $\Lambda$ be a spectrum for $\{0,\ell_1,\ell_2\}$. Define
	\[
	\Sigma\;:=\Z\times\Lambda\subset \R\times\mathbb{T}.
	\]
	Then $\Sigma$ is a spectrum for $\widetilde{\Omega}$. Therefore, $\widetilde{\Omega}$ is both a tile and a spectral set of $\R\times\Z$.
\end{proof}

In what follows, we assume that \(\{0,\ell_1,\ell_2\}\) does not tile \(\mathbb Z\). We first show that if \(\widetilde{\Omega}\) is a weak tile of \(\mathbb R\times\mathbb Z\), then its three unit intervals must satisfy a rigid arithmetic alignment: there exists \(\xi\in\mathbb R\) such that their horizontal positions on the levels \(0,\ell_1,\ell_2\), modulo \(1\), are respectively \(0,\ell_1\xi,\ell_2\xi\). See Figure~\ref{fig:compatible-incompatible-013} for an illustration of this phenomenon. 
\begin{proposition}\label{prop:WT-structure-necessary-condition}
Let $\widetilde{\Omega}$ be as in \eqref{eq:three-intervals-Omega1}, and let $\mu$ be a {non-negative}, locally finite measure on $\R\times\Z$ such that
	\[
	\mathbbm{1}_{\widetilde{\Omega}} * \mu = 1 \qquad \text{a.e. on } \R\times\Z,
	\qquad\text{and}\qquad
	\mu\bigl(\{(0,0)\}\bigr)=1 .
	\]
	Assume moreover that $\{0,\ell_1,\ell_2\}$ is not a tile of $\Z$.
	Then there exists $\xi\in\R$ such that
	\begin{equation}\label{eq:three-intervals-Omega-WT-structure}
		\widetilde{\Om} = (0,1)\times\{0\}+\{(0,0),(m+\ell_1\xi,\ell_1),(n+\ell_2\xi,\ell_2)\}\quad  \text{with}\ m,n\in\Z, \xi\in\R.
	\end{equation}
\end{proposition}
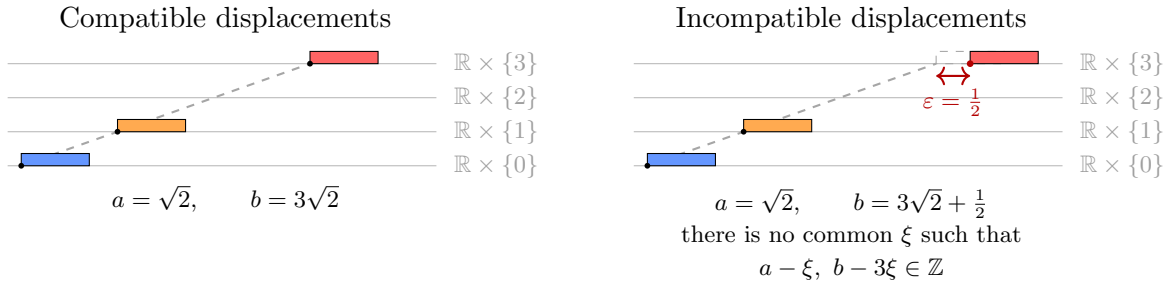
\begin{figure}[!ht]
	\centering
	\begin{tikzpicture}[scale=0.9]
		\pgfmathsetmacro{\xival}{sqrt(2)}
		\pgfmathsetmacro{\wrongshift}{3*sqrt(2)+0.5}
		\pgfmathsetmacro{\h}{0.18}

		\begin{scope}
			\node[font=\small] at (3,2.15) {\textup{Compatible displacements}};

			\foreach \y/\lab in {0/0,0.5/1,1/2,1.5/3}{
				\draw[gray!55,thin] (-0.2,\y)--(6.1,\y);
				\node[right,font=\scriptsize,gray!75] at (6.2,\y)
				{$\mathbb R\times\{\lab\}$};
			}

			\draw[dashed,gray!70,line width=0.8pt]
				(0,0)--({3*\xival},1.5);

			\filldraw[fill=myblue,draw=black]
				(0,0) rectangle (1,\h);
			\filldraw[fill=myorange,draw=black]
				(\xival,0.5) rectangle ({\xival+1},{0.5+\h});
			\filldraw[fill=myred,draw=black]
				({3*\xival},1.5) rectangle ({3*\xival+1},{1.5+\h});

			\fill (0,0) circle (1.2pt);
			\fill (\xival,0.5) circle (1.2pt);
			\fill ({3*\xival},1.5) circle (1.2pt);

			\node[font=\scriptsize,align=center] at (3,-0.48)
				{\(a=\sqrt2,\qquad b=3\sqrt2\)};
		\end{scope}

		\begin{scope}[shift={(9.2,0)}]
			\node[font=\small] at (3,2.15) {\textup{Incompatible displacements}};

			\foreach \y/\lab in {0/0,0.5/1,1/2,1.5/3}{
				\draw[gray!55,thin] (-0.2,\y)--(6.1,\y);
				\node[right,font=\scriptsize,gray!75] at (6.2,\y)
					{$\mathbb R\times\{\lab\}$};
			}

			\draw[dashed,gray!70,line width=0.8pt]
				(0,0)--({3*\xival},1.5);

			\filldraw[fill=myblue,draw=black]
				(0,0) rectangle (1,\h);
			\filldraw[fill=myorange,draw=black]
				(\xival,0.5) rectangle ({\xival+1},{0.5+\h});

			\draw[dashed,gray!70]
				({3*\xival},1.5)
				rectangle ({3*\xival+1},{1.5+\h});

			\filldraw[fill=myred,draw=black]
				(\wrongshift,1.5)
				rectangle ({\wrongshift+1},{1.5+\h});

			\fill (0,0) circle (1.2pt);
			\fill (\xival,0.5) circle (1.2pt);
			\fill[red!70!black] (\wrongshift,1.5) circle (1.4pt);

			\draw[<->,red!70!black,line width=0.8pt]
				({3*\xival},1.32)--(\wrongshift,1.32);
			\node[below,font=\scriptsize,red!70!black]
				at ({(3*\xival+\wrongshift)/2},1.30)
				{$\varepsilon=\frac12$};

			\node[font=\scriptsize,align=center] at (3,-1)
				{$a=\sqrt2,\qquad b=3\sqrt2+\frac{1}{2}$\\there is no common \(\xi\) such that\\
				\(a-\xi,\ b-3\xi\in\mathbb Z\)};
		\end{scope}
	\end{tikzpicture}

	\caption{
		Compatible and incompatible horizontal displacements for
		\(\{0,\ell_1,\ell_2\}=\{0,1,3\}\).
		In the left panel, the three intervals have the common linear
		drift \(\xi=\sqrt2\). In the right panel, the third interval is
		displaced from its compatible position by
		\(\varepsilon=\frac12\). Compatibility of the displacements is only a necessary condition for weak tiling.
	}
	\label{fig:compatible-incompatible-013}
\end{figure}


Let $\mu$ be a weak tiling measure for $\widetilde{\Omega}$. For each $j\in\Z$, recall that the fiber measure $\mu_j$ on $\R$ is defined by
\[
\mu_j(E)\;:=\;\mu(E\times\{j\}),\qquad E\subset\R \text{ Borel}.
\]
By Proposition~\ref{prop:WT-finite-intervals-slicing}, for each $j\in\Z$, we have 
\begin{equation}\label{eq:three-intervals-WT-fiber-measures}
\mathbbm{1}_{(0,1)} * \nu_j = 1
\qquad \text{a.e. on }\R\quad\text{with }\nu_j = \mu_j + \mu_{j-\ell_1}*\delta_a + \mu_{j-\ell_2}*\delta_b.
\end{equation}

The proof of Proposition~\ref{prop:WT-structure-necessary-condition} is based on a fiberwise analysis of these measures. We first derive nonemptiness and support properties of the fiber measures $\mu_j$. We then use the relation defining $\nu_j$ together with the one-dimensional weak tiling identity above to locate atoms on the fibers indexed by $d\Z$, where $d:=\gcd(\ell_1,\ell_2)$. This will yield the desired structural conclusion as in \eqref{eq:three-intervals-Omega-WT-structure}.

Our first step is to show that every fiber $\mu_j$ with $j\in d\Z$ is nontrivial. More precisely, we prove that on suitable translates of $(0,1)$, the weak tiling is carried entirely by a single fiber measure. This is the content of the next lemma, which extends Lemma~\ref{lem:three-points-weak-tiling-measure} to the setting of $\R\times\Z$.

\begin{lemma}\label{lem:nonempty-supp(mu_j)-gcd-fibers}
Let $\widetilde{\Omega}$ be as in \eqref{eq:three-intervals-Omega1}, and let $\mu$ be a {non-negative}, locally finite measure such that
	\[
	\mathbbm{1}_{\widetilde{\Omega}} * \mu = 1 \qquad \text{a.e. on } \mathbb{R}\times\mathbb{Z}.
	\]
	Suppose further that $$\mathbbm{1}_{(0,1)}*\mu_0(x)=1\qae\ x\in(0,1).$$
Then for any integers $m,n$ with $m\equiv n\pmod{3}$, 
	$$\mathbbm{1}_{(0,1)}*\mu_{m\ell_1+n\ell_2}(x)=1\qae\  x\in(0,1)+ma+nb.$$
\end{lemma}
\begin{proof}
For each $m,n\in\Z$, set
\[
I_{m,n}:=(0,1)+ma+nb,
\qquad
\varphi_{m,n}(x):=\mathbbm{1}_{(0,1)}*\mu_{m\ell_1+n\ell_2}(x).
\]
Then each $\varphi_{m,n}$ is nonnegative, and by assumption 
\[
\varphi_{0,0}(x)=1 \qae x\in I_{0,0}=(0,1).
\]
Moreover, by \eqref{eq:three-intervals-WT-fiber-measures}, for every $m,n\in\Z$ we have
\begin{equation}\label{eq:phi-recursion-proof}
\varphi_{m,n}(x)+\varphi_{m-1,n}(x-a)+\varphi_{m,n-1}(x-b)=1
\qae x\in\R.
\end{equation}

We now propagate the values of the functions $\varphi_{m,n}$ across the lattice, exactly as in Figure~\ref{fig:nu-and-propagation}. 
For convenience, write
\[
P(m,n):\ \varphi_{m,n}(x)=1 \qae x\in I_{m,n},
\qquad
Q(m,n):\ \varphi_{m,n}(x)=0 \qae x\in I_{m,n}.
\]
Since all the functions $\varphi_{m,n}$ are nonnegative, restricting
\eqref{eq:phi-recursion-proof} to $I_{m,n}$ gives exactly the same
triangular propagation rule as in the proof of
Lemma~\ref{lem:three-points-weak-tiling-measure}: on each triangle
with vertices $(m,n)$, $(m-1,n)$ and $(m,n-1)$, one occurrence of $P$ forces $Q$ at the other two vertices, while
two occurrences of $Q$ force $P$ at the remaining vertex.

Since $P(0,0)$ holds, the propagation argument in the proof of
Lemma~\ref{lem:three-points-weak-tiling-measure} applies verbatim.
Consequently, $P(m,n)$ holds whenever $m\equiv n\pmod 3$. That is,
\[
	\mathbbm{1}_{(0,1)}*\mu_{m\ell_1+n\ell_2}(x)
	=1
	\qae x\in (0,1)+ma+nb
\]
for every $m,n\in\Z$ satisfying $m\equiv n\pmod 3$.

\end{proof}

Using Lemma~\ref{lem:MN-mod-3-generate-gcd}, we can now extend the previous conclusion from the congruence class condition $m\equiv n\pmod 3$ to all fibers indexed by $d\Z$.

\begin{corollary}\label{cor:nonempty-supp(mu_j)-all-gcd-fibers}
Let $\widetilde{\Omega}$ be as in \eqref{eq:three-intervals-Omega1}, and let $\mu$ be a {non-negative}, locally finite measure such that
	\[
	\mathbbm{1}_{\widetilde{\Omega}} * \mu = 1 \qquad \text{a.e. on } \mathbb{R}\times\mathbb{Z}\quad\text{with }\mu(\{(0,0)\})=1.
	\]
	Suppose further that $\{0,\ell_1,\ell_2\}$ is not a tile of $\Z$. 
Then for any $\ell\in d\Z$, there exist $m^*,n^*\in\Z$ with $m^*\equiv n^*\pmod{3}$ such that 
	$$\mathbbm{1}_{(0,1)}*\mu_{\ell}(x)=1\qae\  x\in(0,1)+m^*a+n^*b.$$
\end{corollary}
\begin{proof}
Since $\nu_0=\mu_0+\mu_{-\ell_1}*\delta_a+\mu_{-\ell_2}*\delta_b$, and everything is nonnegative,
    \[
    \mathbbm{1}_{(0,1)}*\mu_0 \le \mathbbm{1}_{(0,1)}*\nu_0 =1\qae
    \]
But the unit atom $\mu(\{(0,0)\})=1$ gives $\mu_0(\{0\})=1$ and then $\mathbbm{1}_{(0,1)}*\mu_0\geq 1$  a.e. on $(0,1)$. Hence equality holds.
    
Since $\{0,\ell_1,\ell_2\}$ is not a tile of $\Z$, we have 
	\[ \left\{0,\frac{\ell_1}{d},\frac{\ell_2}{d}\right\} \not\equiv \{0,1,2\} \pmod{3}. \]
	By Lemma~\ref{lem:MN-mod-3-generate-gcd}, for each $\ell\in d\Z$, there exist $m^*,n^*\in\Z$ with $m^*\equiv n^*\pmod{3}$ such that $\ell=m^*\ell_1+n^*\ell_2$. The conclusion then follows from Lemma~\ref{lem:nonempty-supp(mu_j)-gcd-fibers}.
\end{proof}

The next step is to strengthen the non-emptiness result by showing that each fiber $\mu_j$ with $j\in d\Z$ actually carries an atom whose location is determined modulo $1$ by the parameters $a$ and $b$.

\begin{lemma}\label{lem:atoms-on-gcd-fibers}
Let $\widetilde{\Omega}$ be as in \eqref{eq:three-intervals-Omega1}, and let $\mu$ be a {non-negative}, locally finite measure on $\R\times\Z$ such that
	\[
	\mathbbm{1}_{\widetilde{\Omega}} * \mu = 1 \qquad \text{a.e. on } \mathbb{R}\times\mathbb{Z}\quad\text{with }\mu(\{(0,0)\})=1.
	\]
	Suppose that $\{0,\ell_1,\ell_2\}$ is not a tile of $\Z$. Then for every $j=m\ell_1+n\ell_2\in d\Z$, there exists $\xi_j=\xi_j(m,n)\in\R$ such that
	\begin{equation}\label{eq:atoms-on-gcd-fibers-linear-combo}
	\mu_j(\{\xi_j\})=1,
	\qquad\text{and}\qquad
	\xi_j \equiv ma+nb \pmod{1}.
	\end{equation}
\end{lemma}
\begin{proof}
Since $\mu(\{(0,0)\})=1$, by the definition of $\mu_0$, we have $\mu_0(\{0\})=1>0$. We first prove that the conclusion holds for the cases $j=\pm\ell_1$ and $j=\pm\ell_2$.

Recall that for each $j\in\Z$, we have \eqref{eq:three-intervals-WT-fiber-measures}, i.e.,
	\begin{equation}\label{eq:three-intervals-WT-fiber-measures-proof-2}
		\mathbbm{1}_{(0,1)} * \nu_j= 1 \quad \text{a.e. on }\R\qquad \text{with }\nu_j=\mu_{j}+\mu_{j-\ell_1}*\delta_a+\mu_{j-\ell_2}*\delta_b.
	\end{equation}
	Since $\mu_0(\{0\})=1$ and by the definition of $\nu_{\ell_1}$,
	we have $$1 \ge \nu_{\ell_1}(\{a\})\ge \mu_0*\delta_a(\{a\})=1$$ and then $\nu_{\ell_1}(\{a\})=1$. By Corollary~\ref{cor:unit-interval-WTmeasure-constant-masses}, we have 
	$$\nu_{\ell_1}(\{a+k\})=1\quad \text{for any } k\in\Z.$$
	Corollary~\ref{cor:nonempty-supp(mu_j)-all-gcd-fibers} implies that there exist integers $m^*_{\ell_1}$ and $n^*_{\ell_1}$ with $m^*_{\ell_1}\equiv n^*_{\ell_1}\pmod{3}$ such that the interval $(0,1)+m^*_{\ell_1}a+n^*_{\ell_1}b$ must be weakly tiled by the fiber measure $\mu_{\ell_1}$ and not by any other fiber measure in the weak tiling equation \eqref{eq:three-intervals-WT-fiber-measures-proof-2} with $j=\ell_1$. Since there exists a unique $k_{\ell_1}\in\mathbb Z$ such that
    \[
    	\xi_{\ell_1}:=a+k_{\ell_1}
    	\in [0,1)+m_{\ell_1}^*a+n_{\ell_1}^*b,
    \]
    and $\mu_{\ell_1}$ is the only fiber measure contributing on this
    interval, the weak tiling equation on the nontrivial overlap of
    $(\xi_{\ell_1},\xi_{\ell_1}+1)$ with
    $(0,1)+m_{\ell_1}^*a+n_{\ell_1}^*b$ gives
    \[
    	\mu_{\ell_1}(\{\xi_{\ell_1}\})
    	=
    	\nu_{\ell_1}(\{\xi_{\ell_1}\})
    	=
    	1,
    	\qquad
    	\xi_{\ell_1}\equiv a\pmod 1.
    \]
    

	Now we show that $\mu_{-\ell_1}$ carries an atom at a location congruent to $-a$ modulo $1$. Setting $j=0$ in \eqref{eq:three-intervals-WT-fiber-measures-proof-2}, we obtain
	\begin{equation}\label{eq:three-intervals-WT-fiber-measures-proof-j=0}
		\mathbbm{1}_{(0,1)} * \nu_0= 1 \qae\quad \text{with }\nu_0=\mu_{0}+\mu_{-\ell_1}*\delta_a+\mu_{-\ell_2}*\delta_b.
	\end{equation}
	Since $\nu_0(\{0\})=1$, by Corollary~\ref{cor:unit-interval-WTmeasure-constant-masses}, we have $\nu_0(\{k\})=1$ for each $k\in\Z.$
	By Corollary~\ref{cor:nonempty-supp(mu_j)-all-gcd-fibers}, for $j=-\ell_1$, there exist integers $m^*_{-\ell_1}$ and $n^*_{-\ell_1}$ with $m^*_{-\ell_1}\equiv n^*_{-\ell_1}\pmod{3}$ such that 
	$$\mathbbm{1}_{(0,1)}*\mu_{-\ell_1}*\delta_a=1\qae\  x\in(0,1)+(m^*_{-\ell_1}+1)a+n^*_{-\ell_1}b,$$
	which implies that the interval $(0,1)+(m^*_{-\ell_1}+1)a+n^*_{-\ell_1}b$ is weakly tiled by the fiber measure $\mu_{-\ell_1}*\delta_a$ and not by any other fiber measure in the weak tiling equation \eqref{eq:three-intervals-WT-fiber-measures-proof-j=0}. Since there exists a unique $k_{-\ell_1}\in\Z$ such that
	$$\xi_{-\ell_1} := k_{-\ell_1}-a\in [0,1)+m^*_{-\ell_1}a+n^*_{-\ell_1}b,$$
	it follows that 
	$$\mu_{-\ell_1}(\{\xi_{-\ell_1}\})=\mu_{-\ell_1}*\delta_a(\{k_{-\ell_1}\})=\nu_0(\{k_{-\ell_1}\})=1.$$
	By the definition of $\xi_{-\ell_1}$, we have $\xi_{-\ell_1}\equiv -a\pmod{1}$.
	
	For $j=\pm\ell_2$, the proof is similar. We omit the details.


Since $\mu_j(\{\xi_j\})=1$, we may translate the whole measure by
$-(\xi_j,j)$. The translated measure still satisfies the same convolution
identity and has a unit atom at the origin. Hence the preceding
four-direction propagation argument applies verbatim.


Now let $j=m\ell_1+n\ell_2$. Starting from the unit atom of $\mu_0$ at $0$ and
iterating the above steps according to this representation, we obtain a unit
atom of $\mu_j$ whose location is congruent to $ma+nb$ modulo $1$. Since
$\ell_1\mathbb Z+\ell_2\mathbb Z=d\mathbb Z$, the conclusion holds for every
$j\in d\mathbb Z$.
\end{proof}

We are now ready to complete the proof of Proposition~\ref{prop:WT-structure-necessary-condition}. The existence of atoms on all fibers in $d\Z$ yields the desired linear congruence relations for $a$ and $b$.
\begin{proof}[Proof of Proposition~\ref{prop:WT-structure-necessary-condition}]
	By Bézout's theorem, there exist integers $m_d,n_d$ such that $m_d\ell_1+n_d\ell_2=d$. By Lemma~\ref{lem:atoms-on-gcd-fibers}, more precisely \eqref{eq:atoms-on-gcd-fibers-linear-combo}, there exists a unit atom $\xi_d$ of $\mu_d$ such that $m_da+n_db\equiv\xi_d\pmod{1}$.

	For any $k\in\mathbb Z$, applying Lemma~\ref{lem:atoms-on-gcd-fibers} to the representation $kd=(km_d)\ell_1+(kn_d)\ell_2$, we obtain a unit atom $\zeta_{kd}$ of $\mu_{kd}$ satisfying
	$$\zeta_{kd}\equiv km_da+kn_db\equiv k\xi_d\pmod{1}.$$

	Since $d\mid\ell_1$, taking $k=\ell_1/d$ gives a unit atom $\zeta_{\ell_1}$ of $\mu_{\ell_1}$ such that
	$\zeta_{\ell_1}\equiv \frac{\ell_1}{d}\xi_d\pmod{1}$. On the other hand, $\mathbbm{1}_{(0,1)}*\nu_{\ell_1}=1$ a.e. on $\mathbb R$ and $\nu_{\ell_1}(\{a\})=1$. Hence, by Corollary~\ref{cor:unit-interval-WTmeasure-constant-masses}, $\nu_{\ell_1}=\delta_{a+\mathbb Z}$. Since $\mu_{\ell_1}\leq\nu_{\ell_1}$, every atom of $\mu_{\ell_1}$ is congruent to $a$ modulo $1$. In particular, $\zeta_{\ell_1}\equiv a\pmod{1}$, and therefore
	\begin{equation}\label{eq:a-linear-combo-xi_d}
		a\equiv \frac{\ell_1}{d}\xi_d\pmod{1}.
	\end{equation}

	Similarly, taking $k=\ell_2/d$ in Lemma~\ref{lem:atoms-on-gcd-fibers}, and using $\mathbbm{1}_{(0,1)}*\nu_{\ell_2}=1$ a.e. on $\mathbb R$ and $\nu_{\ell_2}(\{b\})=1$, we obtain $\nu_{\ell_2}=\delta_{b+\mathbb Z}$ and hence
	\begin{equation}\label{eq:b-linear-combo-xi_d}
		b\equiv \frac{\ell_2}{d}\xi_d\pmod{1}.
	\end{equation}

	Finally, setting $\xi=\xi_d/d$, \eqref{eq:a-linear-combo-xi_d} and \eqref{eq:b-linear-combo-xi_d} give $a-\ell_1\xi\in\mathbb Z$ and $b-\ell_2\xi\in\mathbb Z$. This proves the proposition.
\end{proof}

\subsubsection{Proof of Theorem~\ref{prop-WT-3intervals-characterization}\ \label{subsection7.4} \textup{(3)(ii)}}
Proposition~\ref{prop:WT-structure-necessary-condition} allows us to write $\widetilde{\Om}$ as 
	\begin{equation*}
		\widetilde{\Om} = (0,1)\times\{0\}+\{(0,0),(M+\ell_1\xi,\ell_1),(N+\ell_2\xi,\ell_2)\}\quad  \text{with}\ M,N\in\Z, \xi\in\R.
	\end{equation*}
By Proposition~\ref{prop:translation-invariance}, the shear $T(x,j)=(x-j\xi,j)$ preserves tiling, weak tiling, and spectrality. After applying $T$, we retain the notation $\widetilde{\Omega}$ for the transformed set. Thus it suffices to consider 
\begin{equation}\label{eq:three-intervals-Omega2}
\widetilde{\Omega}=(0,1)\times\{0\}+E,
\end{equation}
where
	\[
	E:=\{(0,0),(M,\ell_1),(N,\ell_2)\}\subset \Z^2.
	\]


{We consider two cases}:
\begin{itemize}
	\item[\textup(1)] \((M,\ell_1)\) and \((N,\ell_2)\) are \(\mathbb Q\)-linearly independent;
	\item[\textup(2)] \((M,\ell_1)\) and \((N,\ell_2)\) are \(\mathbb Q\)-linearly dependent.
\end{itemize}

\begin{lemma}\label{lem:Qlin-indep-notile-tiling}
Let $\widetilde{\Omega}$ be as in \eqref{eq:three-intervals-Omega2}. Suppose that $\{0,\ell_1,\ell_2\}$ is not a tile of $\Z$. If $(M,\ell_1),(N,\ell_2)$ are $\Q$-linearly independent, then $\widetilde{\Omega}$ is both a tile and a spectral set of $\R\times\Z$.
\end{lemma}
\begin{figure}[H]
	\centering
	\begin{tikzpicture}[scale=1]
		\pgfmathsetmacro{\M}{0}
		\pgfmathsetmacro{\N}{-1}
		\pgfmathsetmacro{\ellone}{1}
		\pgfmathsetmacro{\elltwo}{3}

		\foreach \m in {-3,-2,-1,0,1,2,3,4,5}{
			\begin{scope}
				\clip (-3.5,\m/2) rectangle (4,\m/2+0.2);
				\foreach \n in {-5,-4,-3,-2,-1,0,1,2,3,4,5}{
					\pgfmathtruncatemacro{\r}{mod(\n+\m+300,3)}

					\ifcase\r
						\def\thiscolor{myblue}
					\or
						\def\thiscolor{myorange}
					\or
						\def\thiscolor{myred}
					\fi

					\fill[\thiscolor] (\n,\m/2) rectangle ({\n+1},{\m/2+0.2});
					\draw[black] (\n,\m/2) rectangle ({\n+1},{\m/2+0.2});
				}
			\end{scope}

			\draw[thick] (-3.5,\m/2)--(4,\m/2)
			node[right,scale=0.8] {$\mathbb R\times\{\m\}$};
		}

		\draw[white,line width=3pt] (0,0) rectangle (1,0.2);
		\draw[white,line width=3pt] (\M,\ellone/2) rectangle ({\M+1},{\ellone/2+0.2});
		\draw[white,line width=3pt] (\N,\elltwo/2) rectangle ({\N+1},{\elltwo/2+0.2});

		\draw[black,line width=1.5pt] (0,0) rectangle (1,0.2);
		\draw[black,line width=1.5pt] (\M,\ellone/2) rectangle ({\M+1},{\ellone/2+0.2});
		\draw[black,line width=1.5pt] (\N,\elltwo/2) rectangle ({\N+1},{\elltwo/2+0.2});
	\end{tikzpicture}

	\caption{A translational tiling of \(\mathbb R\times\mathbb Z\) in the case where \(0,\ell_1,\ell_2\) are distinct and \(\{0,\ell_1,\ell_2\}\) does not tile \(\mathbb Z\). Here the framed intervals are precisely \(\widetilde{\Omega}\), with \((M,\ell_1)=(0,1)\), \((N,\ell_2)=(-1,3)\), and the translation set is the lattice \(\{(n,m)\in\mathbb Z^2:n+m\equiv0\pmod 3\}\).}
	\label{fig:three-distinct-levels-tiling}
\end{figure}
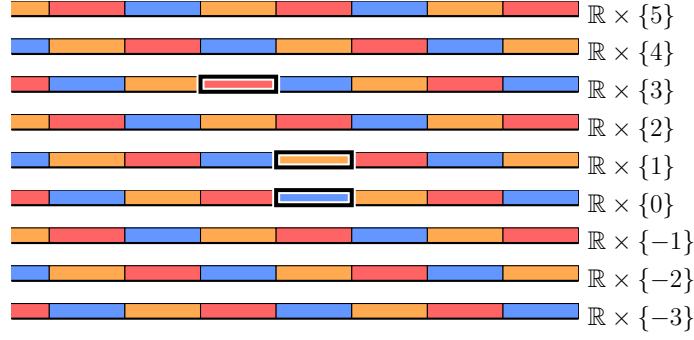

\begin{proof}
By Corollary~\ref{cor:lift-1D-to-2D}, it suffices to find $\eta\in\R$ such that $$E_\eta:=\{0,\,M+\ell_1\eta,\,N+\ell_2\eta\}$$ forms a three-element tile of $\Z$.

Let $d:=\gcd(\ell_1,\ell_2),\ \alpha:=\ell_1/d,\ \beta:=\ell_2/d.$
	Then \(\gcd(\alpha,\beta)=1\). Since \((M,\ell_1)\) and \((N,\ell_2)\) are \(\Q\)-linearly independent,
	\[
	\Delta:=M\beta-N\alpha=\frac{M\ell_2-N\ell_1}{d}\neq 0.
	\]
	Write $\Delta=3^s q,$ where \(s\ge 0\) and \(3\nmid q\).

Since \(\{0,\ell_1,\ell_2\}\) does not tile \(\Z\), the set \(\{0,\alpha,\beta\}\) is not a complete residue system modulo \(3\). Hence the two numbers $\beta-2\alpha,\ 2\beta-\alpha$ represent the two nonzero residue classes modulo \(3\). Therefore we may choose \((r,t)\in\{(1,2),(2,1)\}\) such that
	\[
	\beta r-\alpha t\equiv q \pmod 3.
	\]
	Write $q-(\beta r-\alpha t)=3h$ for some \(h\in\Z\). Since \(\gcd(\alpha,\beta)=1\), there exist \(m,n\in\Z\) such that $\beta m-\alpha n=h.$ Set $u:=r+3m,\ v:=t+3n.$ Then \(u\) and \(v\) represent the two distinct nonzero residue classes modulo \(3\), and $\beta u-\alpha v=q.$

Now define $\eta:=(3^s u-M)/\ell_1.$
	Since
	\[
	\ell_2(3^s u-M)-\ell_1(3^s v-N)
	=
	d\bigl(3^s(\beta u-\alpha v)-(M\beta-N\alpha)\bigr)=0,
	\]
	we also have $\eta=(3^s v-N)/\ell_2$.
	Thus
	\[
	M+\ell_1\eta=3^s u,\qquad N+\ell_2\eta=3^s v\quad\text{and}\quad
	E_\eta=\{0,3^s u,3^s v\}.
	\]
After dividing by its greatest common divisor, the two nonzero elements remain distinct nonzero residues modulo \(3\). Therefore \(E_\eta\) is a tile of \(\Z\). The spectrality of \(\widetilde\Omega\) follows in the same way as in Lemma~\ref{lem:three-intervals-l1=l2-not0}.
\end{proof}

The proof of the following lemma is essentially identical to that of Theorem~\ref{prop-three-intervals-R-tile-Z}.
\begin{lemma}\label{lem:Qlin-dep-notile-no-weak-tiling}
Let $\widetilde{\Omega}$ be as in \eqref{eq:three-intervals-Omega2}. Assume that $\{0,\ell_1,\ell_2\}$ is not a tile of $\Z$, and that $(M,\ell_1),(N,\ell_2)$ are $\Q$-linearly dependent. 
Then $\widetilde{\Omega}$ is not a weak tile.
\end{lemma}

\begin{proof}
Let $\ell=\gcd(\ell_1,\ell_2)$. We argue by contradiction.  
Assume that there exists a {non-negative}, locally finite Borel measure $\mu$ on $\R\times\Z$ with $\mu(\{(0,0)\})=1$ such that
\[
\mathbbm{1}_{\widetilde{\Omega}} * \mu = 1 \qquad \text{a.e. on }\R\times\Z.
\]
Recall that for each $j\in\Z$, \eqref{eq:three-intervals-WT-fiber-measures} gives
\begin{equation}\label{eq:three-intervals-WT-fiber-measures-proof-3}
\mathbbm{1}_{(0,1)} * \nu_j = 1 \quad \text{a.e. on }\R,
\qquad 
\nu_j=\mu_{j}+\mu_{j-\ell_1}*\delta_M+\mu_{j-\ell_2}*\delta_N .
\end{equation}

As in the proof of Corollary~\ref{cor:nonempty-supp(mu_j)-all-gcd-fibers}, the assumptions $\mu(\{(0,0)\})=1$ and $\mathbbm{1}_{\widetilde{\Omega}}*\mu=1$ imply that $\mathbbm{1}_{(0,1)}*\mu_0=1$ a.e. on $(0,1)$. Hence by Lemma~\ref{lem:nonempty-supp(mu_j)-gcd-fibers}, for any integers $s,t$ with $s\equiv t\pmod{3}$, we have
\begin{equation}\label{eq:three-intervals-WT-fiber-measures-linear-combo}
\mathbbm{1}_{(0,1)}*\mu_{s\ell_1+t\ell_2}(x)=1 \quad \text{a.e. for }x\in(0,1)+sM+tN.
\end{equation}
In particular, since $\{0,\ell_1,\ell_2\}$ is not a tile of $\Z$, Lemma~\ref{lem:MN-mod-3-generate-gcd} implies
\[
\ell\Z=\{\,s\ell_1+t\ell_2:\ s,t\in\Z,\ s\equiv t \pmod{3}\,\}.
\]

Next, since $(M,\ell_1),(N,\ell_2)$ are $\Q$-linearly dependent, we have $M\ell_2=N\ell_1.$ Since $\ell_1\ell_2\neq 0$, we may set $c:=M/\ell_1=N/\ell_2$, so that $M=c\ell_1$ and $N=c\ell_2$. Choose integers $s,t$ with $s\equiv t\pmod{3}$ and $s\ell_1+t\ell_2=\ell_1.$ (Such $s,t$ exist because $\ell_1\in \ell\Z$ and $\ell\Z$ is generated as above.)
	Then
	\[
	sM+tN=c(s\ell_1+t\ell_2)=c\ell_1=M.
	\]
	Applying \eqref{eq:three-intervals-WT-fiber-measures-linear-combo} with this pair $(s,t)$ gives
	\begin{equation}\label{eq:mu-l1-tiling-on-shift}
	\mathbbm{1}_{(0,1)}*\mu_{\ell_1}(x)=1 \quad \text{a.e. for }x\in(0,1)+M.
	\end{equation}

Now let $j=\ell_1$ in \eqref{eq:three-intervals-WT-fiber-measures-proof-3}. We obtain
\[
\mathbbm{1}_{(0,1)} *\Big(\mu_{\ell_1}+\mu_{0}*\delta_M+\mu_{\ell_1-\ell_2}*\delta_N\Big)= 1 
\quad \text{a.e. on }\R.
\]
Since $\mu(\{(0,0)\})=1$, the fiber measure $\mu_0$ has an atom of mass $1$ at $0$, and hence $\mu_0*\delta_M$ has an atom of mass $1$ at $M$. Therefore $\nu_{\ell_1}$ has a unit mass at $M$, so for a.e. $x\in(0,1)+M$,
\[
1=\mathbbm{1}_{(0,1)}*\nu_{\ell_1}(x)
\;\ge\;
\mathbbm{1}_{(0,1)}*\big(\mu_{\ell_1}+\delta_M\big)(x)
=
\mathbbm{1}_{(0,1)}*\mu_{\ell_1}(x)+\mathbbm{1}_{(0,1)}*\delta_M(x).
\]
For $x\in(0,1)+M$ we have $\mathbbm{1}_{(0,1)}*\delta_M(x)=1$, and by \eqref{eq:mu-l1-tiling-on-shift} also
$\mathbbm{1}_{(0,1)}*\mu_{\ell_1}(x)=1$ a.e. on $(0,1)+M$. Hence the right-hand side equals $2$ for a.e. $x\in(0,1)+M$, contradicting
$\mathbbm{1}_{(0,1)}*\nu_{\ell_1}(x)=1$ a.e. on $\R$.
\end{proof}


\section{Fuglede’s conjecture for three non-overlapping unit squares}\label{S8}
We first establish a geometric result about tilings of the plane by congruent squares, allowing both translations and rotations. We then use this result to complete the proof of Theorem~\ref{thm:Fuglede-conjecture-3squares}. 
\begin{proposition}\label{lem:congruent-square-tiling-parallel}
    Every tiling of \(\mathbb R^2\) by translated and rotated copies of a square consists of pairwise parallel squares.
\end{proposition}
\begin{proof}
Let \(\mathcal S\) be a tiling of \(\mathbb R^2\) by congruent squares, and let \(V\) denote the set of all vertices of the squares in \(\mathcal S\). Since the tiling is locally finite, \(V\) is locally finite, i.e. all balls of radius 1 have only finitely many points from $V$. 

Take any two squares \(Q,Q'\in\mathcal S\), and choose points \(x\in\operatorname{int}(Q)\) and \(x'\in\operatorname{int}(Q')\). Since \(\mathbb R^2\setminus V\) is path connected, we may join \(x\) to \(x'\) by a polygonal path avoiding \(V\). After a small perturbation, we may further assume that the path crosses every tile boundary transversely and does not run along any tile boundary.

Since the image of the path is compact and the tiling is locally finite, the path meets only finitely many squares. Denote them, in the order in which they are traversed, by \(Q=Q_0,Q_1,\ldots,Q_N=Q'\).

Suppose that the path passes from \(Q_i\) to \(Q_{i+1}\) at a point \(p\). Since \(p\notin V\), it lies in the relative interior of a side of each of \(Q_i\) and \(Q_{i+1}\). We claim that these two sides are collinear. Indeed, if their supporting lines were distinct, they would intersect transversely at \(p\). Near \(p\), each square occupies one side of its corresponding supporting line, and the two resulting half-planes have an open sector in common. It would follow that \(\operatorname{int}(Q_i)\cap\operatorname{int}(Q_{i+1})\neq\varnothing\), contrary to the definition of a tiling.

Thus \(Q_i\) and \(Q_{i+1}\) share a boundary segment of positive length. In particular, one edge direction of \(Q_i\) agrees with one edge direction of \(Q_{i+1}\). Since the other edge direction of a square is perpendicular to the first, the two squares have the same pair of edge directions and are therefore parallel. Consequently, all $Q_0,\cdots, Q_N$ are parallel.  Since \(Q\) and \(Q'\) were arbitrary, all squares in the tiling are pairwise parallel.
\end{proof}

Now we are ready to prove Theorem \ref{th-three-cubes-parallel} in $\R^2$ and Theorem~\ref{thm:Fuglede-conjecture-3squares}.

\begin{proof}[Proof of Theorem~\ref{th-three-cubes-parallel} in $\R^2$] We can write
\begin{equation}\label{-Omega-final-form}
\Omega=(0,1)^2+A \quad \text{with } 0\in A\subset\R^2,\ \#A=3.
\end{equation}

By Corollary~\ref{cor-equal-size-necessary-dimension-2}, without loss of generality, we may further assume that $\pi_2(a)\in\mathbb{Z}$ for each $a\in A$.
Then by Theorem~\ref{thm:layered-reduction}, $\Omega$ is a weak tile of $\mathbb{R}^2$ if and only if the set
$$
\widetilde{\Omega}=\bigcup_{a\in A} ((0,1)+\pi_1(a))\times\{\pi_2(a)\}
$$
is a weak tile of $\mathbb{R}\times\mathbb{Z}$. Theorem~\ref{prop-WT-3intervals-characterization} implies that if $\widetilde{\Omega}$ is a weak tile, then it is both a tile and a spectral set of $\mathbb{R}\times\mathbb{Z}$. Hence by Theorem~\ref{thm:layered-reduction}, $\Omega$ is both a tile and a spectral set of $\mathbb{R}^2$. This completes the proof.
\end{proof}
\begin{proof}[Proof of Theorem~\ref{thm:Fuglede-conjecture-3squares}]
If the squares are not all parallel, then  $\Omega$ cannot be a tile by Proposition~\ref{lem:congruent-square-tiling-parallel}, and it cannot be a spectral set either. Indeed, if the three squares are not all parallel, one of them can be chosen as the standard unit square so that the other two are not parallel to its coordinate directions. This also covers the case where two squares are parallel and the third is not: after a rotation and a translation, we may take the latter square to be $Q=(0,1)^2$. Then neither of the other two squares has a facet normal equal to $\pm e_1$ or $\pm e_2$, and hence Theorem~\ref{th-spectral-non-parallel} shows that $\Omega$ is not a spectral set. 

It remains to consider the case of three
parallel squares. After a rotation and translation, $\Om$ has the form (\ref{-Omega-final-form}). The conclusion follows immediately from Theorem \ref{th-three-cubes-parallel}. 
\end{proof}

\noindent{\bf Disclosure of AI use.} The proofs of Proposition~\ref{prop:tiling-layered-reduction}, Proposition~\ref{prop:weak-tiling-layered-reduction} and Proposition \ref{lem:congruent-square-tiling-parallel} were first generated by ChatGPT 5.6-Sol and the authors rewrote them with more clarity. The other proofs were written by human. The authors take full responsibilities for the correctness of the manuscript.

\bibliographystyle{plain}
\bibliography{references.bib}

\end{document}